\documentclass{amsart}

\usepackage{latexsym,enumerate}
\usepackage{amsmath,amsthm,amsopn,amstext,amscd,amsfonts,amssymb}
\usepackage[ansinew]{inputenc}
\usepackage{verbatim}
\usepackage{graphicx}
\usepackage{pstricks}
\usepackage{wasysym}
\usepackage[all]{xy}
\usepackage{epsfig} 
\usepackage{graphicx,psfrag} 
\usepackage{subfigure}
\usepackage{pstricks,pst-node}

\usepackage{multicol}
\usepackage{color}
\usepackage{colortbl}

\newcommand{\NN}{\mathbb{N}}

\newcommand{\RR}{\mathbb{R}}
\newcommand{\ZZ}{\mathbb{Z}}

\newtheorem{theorem}{Theorem}[section]
\newtheorem{lemma}[theorem]{Lemma}
\newtheorem{proposition}[theorem]{Proposition}
\newtheorem{corollary}[theorem]{Corollary}
\newtheorem{definition}[theorem]{Definition}
\newtheorem{example}[theorem]{Example}
\newtheorem{remark}[theorem]{Remark}

\DeclareMathOperator{\diam}{diam}
\DeclareMathOperator{\Vol}{Vol}

\newcommand{\spb}[1]{\smallskip}
\newcommand{\mpb}[1]{\medskip}
\newcommand{\bpb}[1]{\bigskip}

\newcommand{\p}{\partial}

\renewcommand{\a}{\alpha}
\renewcommand{\b}{\beta}
\newcommand{\e}{\varepsilon}
\renewcommand{\d}{\delta}
\newcommand{\D}{\Delta}
\newcommand{\g}{\gamma}
\newcommand{\G}{\Gamma}

\renewcommand{\l}{\lambda}

\newcommand{\s}{\sigma}

\begin{document}
\DeclareGraphicsExtensions{.jpg,.pdf,.mps,.png}

\title{Cheeger isoperimetric constant of Gromov hyperbolic manifolds and graphs}

\author[Alvaro Mart\'{\i}nez-P\'erez]{Alvaro Mart\'{\i}nez-P\'erez$^{(1)}$}
\address{ Facultad CC. Sociales de Talavera,
Avda. Real Fábrica de Seda, s/n. 45600 Talavera de la Reina, Toledo, Spain}
\email{alvaro.martinezperez@uclm.es}
\thanks{$^{(1)}$ Supported in part by a grant
from Ministerio de Econom{\'\i}a y Competitividad (MTM 2012-30719), Spain.
}

\author[Jos\'e M. Rodr{\'\i}guez]{Jos\'e M. Rodr{\'\i}guez$^{(2)}$}
\address{Departamento de Matem\'aticas, Universidad Carlos III de Madrid,
Avenida de la Universidad 30, 28911 Legan\'es, Madrid, Spain}
\email{jomaro@math.uc3m.es}
\thanks{$^{(2)}$ Supported in part by two grants
from Ministerio de Econom{\'\i}a y Competitividad (MTM 2013-46374-P and MTM 2015-69323-REDT), Spain, and a grant from CONACYT (FOMIX-CONACyT-UAGro 249818), M\'exico.}

\date{\today}


\begin{abstract}
In this paper we study the relationship of hyperbolicity and (Cheeger) isoperimetric inequality in the context of Riemannian manifolds and graphs.
We characterize the hyperbolic manifolds and graphs (with bounded local geometry) verifying this isoperimetric inequality, in terms of their Gromov boundary.
Furthermore, we characterize the trees with isoperimetric inequality (without any hypothesis).
As an application of our results, we obtain the solvability of the Dirichlet problem at infinity for these Riemannian manifolds and graphs,
and that the Martin boundary is homeomorphic to the Gromov boundary.
\end{abstract}

\maketitle{}



\section{Introduction}

In this paper we study the relationship of hyperbolicity and (Cheeger) isoperimetric inequality in the context of Riemannian manifolds and graphs
with bounded local geometry.
Cao proved in \cite{Cao} that hyperbolicity with an extra hypothesis on the Gromov boundary implies (Cheeger) isoperimetric inequality
(an extra hypothesis is necessary, since there exist hyperbolic graphs without isoperimetric inequality, as the Cayley graph of the group $\ZZ$).
Furthermore, Example \ref{exa1} and \cite[Section 4, Example (a)]{RT3} show that isoperimetric inequality does not imply hyperbolicity for graphs and manifolds, respectively.

It is natural to look for relations between hyperbolicity and Cheeger isoperimetric inequality,
since Gromov hyperbolicity can be defined in an alternative way by using a different kind of isoperimetric inequality \cite{ABCD}, \cite{G1}.

Isoperimetric inequalities are of interest in pure and applied mathematics (see, e.g., \cite{C2}, \cite{Po}).
There are close connections between isoperimetric inequality and some conformal invariants of Riemannian manifolds and graphs, namely
Poincar\'e-Sobolev inequalities,
the bottom of the spectrum of the Laplace-Beltrami operator, the exponent of convergence,
and the Hausdorff dimensions of the sets of both bounded geodesics and escaping geodesics in a negatively curved surface
(see \cite{BJ}, \cite[p.228]{Bu}, \cite{Ch}, \cite{FM1}, \cite{FM2}, \cite{FMP1}, \cite{FMP2}, \cite{FR1}, \cite{MRT}, \cite{P}, \cite[p.333]{S}).
The Cheeger isoperimetric inequality is closely related to the project of Ancona
on the space of positive harmonic functions of Gromov-hyperbolic manifolds and graphs (\cite{A1}, \cite{A2} and \cite{A3}).
In fact, in the study of the Laplace operator on a hyperbolic manifold or graph $X$,
Ancona obtained in these three last papers interesting results, under the additional assumption that the bottom of the spectrum of the Laplace spectrum $\l_1(X)$ is positive.
The well-known Cheeger inequality $\l_1(X) \ge \frac14 \,h(X)^2$, where $h(X)$ is the isoperimetric constant of $X$,
guarantees that $\l_1(X)>0$ when $h(X) > 0$ (see \cite{Bu1} for a converse inequality).
Hence, the results of this paper are useful in order to obtain these Ancona's results.

Given any Riemannian $n$-manifold $M$, the Cheeger isoperimetric constant of $M$ is defined as
\[h(M) = \inf_A \frac{\Vol_{n-1}(\partial A)}{\Vol_{n}(A)} ,\]
where $A$ ranges over all non-empty bounded open subsets of $M$,
and $\Vol_{k}(B)$ denotes the $k$-dimensional Riemannian volume of the set $B$.

Given any graph $\G=(V,E)=(V(\G),E(\G))$, let us consider the natural length metric $d_{\,\Gamma}$ where every edge has length 1.
For any graph $\Gamma$, any vertex $v\in V$ and any $k\in \mathbb{N}$, let $S(v,k):=\{w\in V \, | \, d_{\,\Gamma}(v,w)=k \}$.
As usual, we denote by $B(v,k)$ and $\bar{B}(v,k)$ the open and closed balls, respectively.

The combinatorial Cheeger isoperimetric constant of $\Gamma$ is defined to be
\[h(\Gamma) = \inf_A \frac{|\partial A|}{|A|} ,\]
where $A$ ranges over all non-empty finite subsets of vertices in $\Gamma$,
$\partial A = \{v \in \Gamma \, | \, d_{\,\Gamma}(v,A) = 1\}$ and $|A|$ denotes the cardinality of $A$.

We say that a Riemannian manifold or graph $X$ satisfies the (Cheeger) \emph{isoperimetric inequality} if $h(X)>0$, since in this case
$$
\Vol_{n}(A) \le h(X)^{-1} \Vol_{n-1}(\partial A) ,
$$
for every bounded open set $A \subseteq X$ if $X$ is a Riemannian $n$-manifold, and
$$
|A| \le h(X)^{-1} |\partial A| ,
$$
for every finite set $A \subseteq V(X)$ if $X$ is a graph.

Along the paper, we just consider manifolds and graphs $X$ which are connected.
This is not a loss of generality, since if $X$ has connected components $\{X_j\}$, then $h(X)=\inf_j h(X_j)$.

Let $(X,d_X)$ and $(Y,d_Y)$  be two metric spaces. A map $f: X\longrightarrow Y$ is said to be
an $(\alpha, \beta)$-\emph{quasi-isometric embedding}, with constants $\alpha\geq 1,\
\beta\geq 0$, if for every $x, y\in X$:
$$
\alpha^{-1}d_X(x,y)-\beta\leq d_{\,Y}(f(x),f(y))\leq \alpha \, d_X(x,y)+\beta.
$$
The function $f$ is $\varepsilon$-\emph{full} if
for each $y \in Y$ there exists $x\in X$ with $d_Y(f(x),y)\leq \varepsilon$.

A map $f: X\longrightarrow Y$ is said to be
a \emph{quasi-isometry}, if there exist constants $\alpha\geq 1,\
\beta,\varepsilon \geq 0$ such that $f$ is an $\varepsilon$-full
$(\alpha, \beta)$-quasi-isometric embedding.
Two metric spaces $X$ and $Y$ are \emph{quasi-isometric} if there exists
a quasi-isometry $f:X\longrightarrow Y$.
One can check that to be quasi-isometric is an equivalence relation.

A graph $\Gamma$ is said to be $\mu$-\emph{uniform} if each vertex $p$ of $V$ has at most $\mu$ neighbors, i.e.,
\[\sup\big\{|N(p)| \,  \big| \,\, p\in V(\G)\big\}\leq \mu.  \]
If a graph $\Gamma$ is $\mu$-uniform for some constant $\mu$ we say that $\Gamma$ is \emph{uniform}
or that it has \emph{bounded local geometry}.

A Riemannian $n$-manifold $M$ has \emph{bounded local geometry} if there exist positive constants $r,c,$ such that for every $x \in M$ there is a diffeomorphism $F : B(x,r) \rightarrow \RR^n$
with
$$
\frac{1}{c}\,d(x_1,x_2) \le \| F(x_1) - F(x_2) \| \le c\,d(x_1,x_2)
$$
for every $x_1,x_2 \in B(x,r)$.

The \emph{injectivity radius} inj$(x)$ of \emph{$x\in M$} is defined as the supremum of those $r>0$ such that $B(x,r)$ is simply connected or, equivalently,
as half the infimum of the lengths of the (homotopically non-trivial) loops based at $x$.
The \emph{injectivity radius} inj$(M)$ \emph{of $M$} is the infimum over $x\in M$ of inj$(x)$.

\begin{remark} \label{r:Ricci}
If $M$ has positive injectivity radius and a lower bound on its Ricci curvature, then $M$ has bounded local geometry \cite{AC}.
\end{remark}

A celebrated theorem of Kanai in \cite{K} states that quasi-isometries preserve isoperimetric inequalities
between Riemannian manifolds and graphs with bounded local geometry.
This result also holds with weaker hypotheses in the context of Riemann surfaces \cite{CGPR}, \cite{GPPRT}.
The following lemma is a particular case of Lemma 4.2 in \cite{K}:

\begin{lemma} \label{l:quasi-isometric} Let $\Gamma, \Gamma'$ be a pair of uniform graphs.
If $\Gamma, \Gamma'$ are quasi-isometric, then $h(\Gamma)>0$ if and only if $h(\Gamma')>0$.
\end{lemma}

Let $X$ be a metric space. Fix a base point $o\in X$ and for
$x,x'\in X$ let
$$(x|x')_o=\frac{1}{2}\big(d(x,o)+d(x',o)-d(x,x')\big).$$
The number $(x|x')_o$ is non-negative and it is called the \emph{Gromov product} of $x,x'$ with respect to $o$.

\begin{definition} A metric space $X$ is \emph{(Gromov)
hyperbolic} if it satisfies the $\delta$-inequality
\[(x|y)_o\geq \min\{(x|z)_o,(z|y)_o\}-\delta\] for some $\delta\geq
0$, for every base point $o\in X$ and all $x,y,z \in X$.
\end{definition}

We denote by $\d(X)$ the sharp hyperbolicity constant of $X$:
$$
\d(X)= \sup \Big\{ \min\{(x|z)_o,(z|y)_o\} - (x|y)_o \,\big| \;\, x,y,z,o \in X \Big\}.
$$
Hence, $X$ is hyperbolic if and only if $\d(X)<\infty$.

The theory of Gromov hyperbolic spaces was introduced by M. Gromov for
the study of finitely generated groups (see \cite{G1}).
The concept of Gromov hyperbolicity grasps the essence of negatively curved
spaces like the classical hyperbolic space, Riemannian manifolds of
negative sectional curvature bounded away from $0$, and of discrete spaces like trees
and the Cayley graphs of many finitely generated groups. It is remarkable
that a simple concept leads to such a rich
general theory (see \cite{ABCD, GH, G1}).
This theory has been developed from a geometric point of view to the extent of making
hyperbolic spaces an important class of metric spaces to be studied on their
own (see, e.g., \cite{BH,BBI,BS,GH,V}). In the last years, Gromov hyperbolicity
has been intensely studied in graphs (see, e.g.,
\cite{BRS,BRST,BHB1,K21,K22,M,RSVV,Sha1,Sha2,Sha3,WZ} and the references therein).
Gromov hyperbolicity, specially in graphs, has found applications in different areas such as phylogenetics (see
\cite{DHHKMW,DMT}), real networks (see \cite{AAD,ASM,CMN,CoCoLa,KPKVB,MoSoVi}) or the secure transmission
of information and virus propagation on networks (see \cite{K21,K22}).

We want to remark that the main examples of hyperbolic graphs are the trees.
In fact, the hyperbolicity constant of a metric space can be viewed as a measure of
how ``tree-like'' the space is, since those spaces $X$ with $\delta(X) = 0$ are precisely the metric trees.
This is an interesting subject since, in
many applications, one finds that the borderline between tractable and intractable
cases may be the tree-like degree of the structure to be dealt with
(see, e.g., \cite{CYY}).

\smallskip

In this paper we characterize the trees with isoperimetric inequality in Theorem \ref{t:iitreefinal}, in terms of their Gromov boundary.
The main idea is that a geometric object is usually much simpler near infinity: if one looks at it on the boundary, then its essential features are captured whereas all the background noise faints.
If we consider hyperbolic graphs instead of trees, Theorem \ref{t:iigraph} characterizes the uniform hyperbolic graphs with isoperimetric inequality.
In Theorem \ref{th: isoperimetric_manifolds}, we extend this result for a large class of Riemannian manifolds.
Theorems \ref{t:iigraph} and \ref{th: isoperimetric_manifolds} generalize Cao's Theorem (see Remark \ref{r:Cao}).
Corollaries \ref{c1} and \ref{c2} show that for many manifolds and graphs $X$,
the Dirichlet problem at infinity is solvable on $X$ and the Martin boundary of $X$ is homeomorphic to its Gromov boundary.
Finally, in Section 6 we give another sufficient condition for isoperimetric inequality by using local information of the graph.
It allows to show that there exists a large class of non-hyperbolic graphs with positive Cheeger isoperimetric constant.

\smallskip

Amenability is an important property in the context of geometric group theory. It is well-known that a finitely generated group is amenable if and only its Cayley graph with respect to any finite generating set is amenable.
M. Gromov proved in \cite{G1} that a finitely generated infinite hyperbolic group is amenable if and only if it is virtually cyclic, in which case its boundary contains two points.  There are many properties which are proved to be equivalent to amenability in different contexts. In the context of connected uniform graphs see, for example, \cite{Kap}. As it was mentioned by I. Kapovich in this work (see the references therein), non-amenable graphs play an important role in the study of various probabilistic phenomena, such as random walks, harmonic
analysis, Brownian motion and percolations, on graphs and manifolds.  As it was proved in \cite[Theorem 51]{CGH}, if $\Gamma$ is a connected uniform graph, then $h(\G)>0$ if and only if $\G$ is non-amenable. Therefore, the characterization in Theorem \ref{t:iigraph} can be seen as a characterization of amenability on uniform hyperbolic graphs in terms of the boundary.

\section{Some previous results}

Given a real-valued function $f$ on the vertex set $V$ of any graph $\Gamma$, a discrete version of the gradient can be defined as
\[\nabla_{\!xy} f=f(y)-f(x)\]
for every ordered pair of vertices $x,y$ with $xy \in E$, and $\nabla_{\!xy} f=0$ if $xy \notin E$.

The discrete version of the Laplacian can be also defined as follows:
\[(\Delta{f})(x)=\frac{1}{|N(x)|}\sum_{y\in N(x)}\Big(f(y)-f(x)\Big)=\frac{1}{|N(x)|}\sum_{y\in N(x)}\nabla_{\!xy}f.\]

Let us recall the discrete version of Green's formula (see \cite {CGY}).
If $f$ and $g$ are functions on $V$ and one of them has a finite support, then
\[\sum_{x\in V} (\Delta f)(x) g(x)|N(x)|=-\frac{1}{2}\!\sum_{x,y\in V}(\nabla_{\!xy}f)(\nabla_{\!xy}g).\]

Consider $g=\aleph_A$, the characteristic function of $A$:
\[\aleph_A(x) =\left\{ \begin{tabular}{l}
$1, \quad \mbox{if } x \in A,$\\
$0, \quad \mbox{if } x \notin A.$\end{tabular}
\right.\]

Then, Green's formula implies the following, which appears as Proposition 2.3 in \cite{Cao}.
Since there is a slight difference in the notation we include the proof.
Furthermore, the argument in this proof gives Corollary \ref{c:sufficient} below, which will be useful in the proof of Lemma \ref{l1:caract_complete}.

\begin{proposition}\label{p:sufficient} Let $\Gamma$ be a $\mu$-uniform graph. Suppose that there is a function $f$
defined on the vertex set $V$ and satisfying the following:
\begin{itemize}
	\item[(i)] $|\nabla_{\!xy} f|\le c_1$ for every $xy\in E$,
	\item[(ii)] $(\Delta f)(x) \ge c_2 > 0$ for every $x\in V$.
\end{itemize}	
Then $h(\Gamma) \ge \frac{c_2}{\mu c_1}> 0$.
\end{proposition}

\begin{proof} For any finite subset $A\subset V$,
\[
\begin{aligned}
c_2|A|
& \le \sum_{x\in A} (\Delta f)(x)|N(x)|
=\sum_{x\in V} (\Delta f)(x)\aleph_A(x)|N(x)|
= -\frac{1}{2}\!\sum_{x,y\in V}(\nabla_{\!xy} f)(\nabla_{\!xy}\aleph_A)
\\
& \le \frac{1}{2}\,c_1\!\!\!\sum_{x,y\in V}|\nabla_{\!xy} \aleph_A|
= c_1\!\!\sum_{x\in \partial A}|N(x)\cap A|
\le c_1 \mu |\partial A |.
\end{aligned}
\]
Therefore, $\frac{|\partial A|}{|A|}\geq \frac{c_2}{\mu c_1}$.
\end{proof}

The argument in the proof of Proposition \ref{p:sufficient} has the following direct consequence.

\begin{corollary} \label{c:sufficient} Let $\Gamma$ be a graph. Suppose that there is a function $f$
defined on the vertex set $V$ and satisfying the following:
\begin{itemize}
	\item[(i)] $|\nabla_{\!xy} f|\le c_1$ for every $xy\in E$,
	\item[(ii)] $(\Delta f)(x) \ge c_2 > 0$ for every $x\in V$.
\end{itemize}	
Then $\frac{|\partial A|}{|A|}\geq \frac{c_2}{c_1}$ for every finite subset $A\subset V$ such that $|N(x)\cap A|=1$ for each $x \in \p A$.
\end{corollary}

\section{Trees}

The end space of a tree is one of the ways to define a metric on its boundary at infinity. This definition allows to build several categorical equivalences between trees and ultrametric spaces. Let us recall some basic definitions from \cite{Hug}. See also \cite{Hug-M-M} and \cite{M-M}.

A \emph{rooted tree}, $(T,v)$, consists of a tree $T$ and a fixed point $v\in T$, called  the \emph{root}.

\begin{definition} An \emph{ultrametric space} is a metric space $(X,d)$ such that
$d(x,y)\leq \max \{d(x,z),d(z,y)\}$
for all $x,y,z\in X$.
\end{definition}

\begin{definition}\label{end} The \emph{end space} of a rooted tree $(T,v)$ is given by:
$$end(T,v)=\{F: [0,\infty) \rightarrow T \ |\ \text{$F(0)=v$ and $F$
is an isometric embedding}\}.$$
Let $F, G\in end(T,v)$.
\begin{enumerate}
	\item  The {\it Gromov product at infinity} is $(F|G)_v :=\sup \{t\geq 0 \ |\ F(t)=G(t)\}$.
	\item  The {\it end space metric} is $d_v(F,G) := e^{-(F|G)_v}$.
\end{enumerate}
\end{definition}

\begin{proposition} If $(T,v)$ is a rooted tree, then $(end(T,v),d_v)$ is a complete ultrametric space of diameter
at most $1$.
\end{proposition}

Given a tree $T$ and a fixed point $v$, for any point $x\in T$ let \[T^v_x:=\{y\in T \, | \, x\in [vy]\}\] where $[vy]$ denotes the geodesic in $T$ joining $v$ and $y$. Notice that $T^v_x$ is also connected and, therefore, a tree.

\begin{definition} A rooted tree $(T,v)$ is \emph{geodesically complete} if every isometric embedding
$f: [0,t]\rightarrow T$ with $t>0$ and  $f(0)=v$ extends to an isometric embedding $F: [0,\infty) \rightarrow T$.
\end{definition}

\begin{theorem}\cite{M-M} If $(T,v)$ is a rooted tree, then
there exists a unique geodesically complete subtree $(T_\infty,v)\subseteq (T,v)$ that is maximal under inclusion.
\end{theorem}

\begin{proposition}\label{p:end_complete} If $(T,v)$ is a rooted tree, then $end(T_\infty,v)=end(T,v)$.
\end{proposition}

\begin{proof} Since $(T_\infty,v)\subseteq (T,v)$, it is trivial that $end(T_\infty,v)\subseteq end(T,v)$.

Seeking for a contradiction, assume that there is a geodesic ray $F$ emanating from $v$ in $(T,v)$ with $F[0,\infty)\not \subset T_\infty$.
Define $T':=T_\infty\cup F[0,\infty)$. Notice that $T'$ is also
geodesically complete and $T_\infty \subsetneq T'\subseteq T$ leading to contradiction since $T_\infty$ is maximal with this condition.
Hence, $end(T,v) \subseteq end(T_\infty,v)$.
\end{proof}

Recall that a \emph{geodesic space} is a metric space such that for every couple of points there exists a geodesic joining them.

\begin{definition}  A geodesic space $X$ has a \emph{pole} in a point $v$ if there
exists $M > 0$ such that each point of $X$ lies in an $M$-neighborhood of some geodesic
ray emanating from $v$.
\end{definition}

\begin{proposition}\label{p:quasi-isometry} If $v$ is a pole in a tree $T$, then the inclusion
$i:(T_\infty,v)\to (T,v)$ is a quasi-isometry.
\end{proposition}

\begin{proof}
The inclusion $i$ is a $(1,0)$-quasi-isometric embedding.
Since $v$ is a pole, there is a constant $M>0$ such that $i$ is $M$-full. Hence, it is a quasi-isometry.
\end{proof}

\begin{definition} Given a tree $T$ with a pole $v$ and a natural number $K$, we say that $(T,v)$
is $K$-\emph{pseudo-regular} if for every $t\in \mathbb{N}$
and every $a\in S(v,t)$, there exist at least two points $y_1,y_2$ in $S(v,t+K)\cap T^v_a$.
We say that $(T,v)$ is \emph{pseudo-regular} if there exists some $K$ such that $(T,v)$ is
$K$-pseudo-regular.
\end{definition}

\begin{proposition} If the rooted tree $(T,v)$ is pseudo-regular, then it is geodesically complete.
\end{proposition}

\begin{proof} Choose $K\in\NN$ such that $(T,v)$ is $K$-pseudo-regular. Consider any isometric embedding $F_1: [0,t_0]\rightarrow T$ with $t_0\in \NN$ and $F_1(0)=v$, and let $a_1=F_1(t_0)$.
Since $(T,v)$ is $K$-pseudo-regular, there exists a point $a_2$ in $S(v,t_0+K)\cap T^v_{a_1}$ and there is an extension of $F_1$ to a geodesic $F_2: [0,t_0+K] \to T$ with $F_2(t_0+K)=a_2 \in S(v,t_0+K)$.
Repeating the process we obtain geodesics $F_n:[0,t_0+(n-1)K]\to T$ with $F_n(t_0+(n-1)K)=a_n \in S(v,t_0+(n-1)K)$ where $F_n$ is an extension of $F_{n-1}$.
Then, let $F:[0,\infty) \to T$ be such that $F(t):=F_n(t)$ for any $n$ with $t_0+(n-1)K\geq t$. Trivially, $F$ is an isometric embedding extending $F_1$.
\end{proof}

Given a graph $\G$ and $A\subset V(\G)$, we denote by $G(A)$ the subgraph induced by $A$, i.e., the subgraph of $\G$ with $V(G(A))=A$ and
$E(G(A))=\{ xy\in E(\G)\,|\; x,y\in A\}$.

\begin{lemma} \label{l1:caract_complete} Let $(T,v)$ be a $1$-pseudo-regular rooted tree,
and $A\subset V(T)$ a finite set such that $G(A)$ is connected.
Then $\frac{|\partial A|}{|A|}\geq \frac{1}{3}$.
\end{lemma}

\begin{proof}
Let $f$ be the function defined on $V(T)$ such that $f(w)=n$ for every $w\in S(v,n)$ and $n\ge 0$.
Since $T$ is a tree, it is immediate to check that
$|\nabla_{\!xy}f|= 1$ for every $xy \in E(T)$. Also,
$$
(\Delta f)(v)
= \frac1{|N(v)|}\sum_{w\in N(v)}\big( f(w)-f(v)\big)
= \frac1{|N(v)|}\sum_{w\in N(v)}\big( 1-0\big)
=1.
$$
Since $(T,v)$ is $1$-pseudo-regular, if $w\in S(v,n)$ with $n>0$, then
$$
(\Delta f)(w)
= \frac1{|N(w)|}\big( |N(w)\cap S(v,n+1)|-1 \big)
= \frac{|N(w)\cap S(v,n+1)|-1}{|N(w)\cap S(v,n+1)|+1}
\ge \frac13\,.
$$
Since $T$ is a tree and $G(A)$ is connected, we have $|N(x)\cap A|=1$ for every $x\in \p A$.
Therefore, by Corollary \ref{c:sufficient}, $\frac{|\partial A|}{|A|}\geq \frac{1}{3}$.
\end{proof}

Given a rooted tree $(T,v)$ and $A\subset V(T)$, we say that $w\in\p A$ is a \emph{non-essential vertex in} $\p A$ if $d(v,w)<d(v,z)$ for every $z\in A$,
and we denote by $\p_{ne}A$ the set of non-essential vertices in $\p A$.
Since $T$ is a tree, if $G(A)$ is connected, then $|\p_{ne}A|= 1$ when $v\notin A$ and $\p_{ne}A= \emptyset$ otherwise.
The \emph{essential boundary} $\p_e A$ of $A$ is the set $\p_{e}A = \p A \setminus \p_{ne}A$.
Define $\p_{e}^1 A$ as the set of vertices in $A$ at distance $1$ from $\p_{e}A$, i.e.,
the set of neighbors of $\p_{e}A$ in $A$.

\begin{lemma} \label{l2:caract_complete} Let $(T,v)$ be a $1$-pseudo-regular rooted tree,
and $A\subset V(T)$ a finite set such that $G(A)$ is connected.
Then $\frac{|\partial_e A|}{|A|}\geq \frac{1}{3}$ and $\frac{|\partial_e^1 A|}{|A|}\geq \frac{1}{6}$.
\end{lemma}

\begin{proof}
Let us prove the first inequality.
If $\partial_{ne} A = \emptyset$, then $\partial_{e} A = \partial A$ and Lemma \ref{l1:caract_complete} gives $\frac{|\partial_e A|}{|A|}\geq \frac{1}{3}$.
Assume now $\partial_{ne} A \neq \emptyset$, and thus there is $v_0\in V(T)$ with $\partial_{ne} A=\{v_0\}$.
Since $G(A)$ is connected, $A$ is contained in a connected component $T'$ of $T \setminus \{v_0\}$.
If $v':=N(v_0) \cap V(T')$, then $(T',v')$ is a $1$-pseudo-regular rooted tree.
Applying Lemma \ref{l1:caract_complete} to $A\subset V(T')$, we obtain
$$
\frac{1}{3}
\le \frac{|\partial_{T'} A|}{|A|}
= \frac{|\partial A \setminus \{v_0\}|}{|A|}
= \frac{|\partial_e A|}{|A|}\,.
$$

We prove now the second inequality.
Let $\partial_e^1 A = \{a_1,\dots,a_r\}$ and $T_0=\cup_{j=1}^r [va_j]$; thus, $A\subseteq V(T_0)$.
For each $1\le j \le r$, consider a geodesically complete rooted tree $(T_j,v_j)$ with $|N(v_j)|=2$ and $|N(v)|=3$ for every $v\in V(T_j)\setminus \{v_j\}$.
Let $(T'',v)$ be the rooted tree obtained from $T_0,T_1,\dots,T_r,$ by identifying $a_j$ with $v_j$ for each $1\le j \le r$.
By construction, $(T'',v)$ is a $1$-pseudo-regular rooted tree.
Since $|N(a_j)|=3$ for each $1\le j \le r$, we have $|\partial_{e,T''} A|=2r=2|\partial_e^1 A|$, and
$$
\frac{1}{3}
\le \frac{|\partial_{e,T''} A|}{|A|}
= \frac{2|\partial_e^1 A|}{|A|}\,.
$$
\end{proof}

The following fact is elementary.

\begin{lemma} \label{l4:caract_complete} Let $(T,v)$ be a geodesically complete rooted tree, $A_0\subset V(T)$ a finite set with $G(A_0)$ connected and $\max \{ d(a,v) \,|\; a \in A_0 \} < K$.
Then $|A_0|\leq 1+(K-1)|\p_e A_0|\leq K|\p_e A_0|$.
\end{lemma}

\begin{proposition}\label{p:caract_complete} If $(T,v)$ is a geodesically complete rooted tree, then $h(T)>0$ if and only if $(T,v)$ is pseudo-regular.
Furthermore, if $(T,v)$ is $K$-pseudo-regular, then $h(T)\geq \frac{1}{7K}$.
\end{proposition}

\begin{proof} Suppose $T$ is $K$-pseudo-regular.
Let us define $S_n=S(v,n)$ for each $n\ge 0$.
Since $(T,v)$ is geodesically complete, we have $|S_n| \le |S_{n+1}|$.

Fix a finite set $A\subset V(T)$ with $G(A)$ connected.
By simplicity in the notation, assume that $v\in A$.
Let us define
$$
\begin{aligned}
n_0
& := \max \big\{ n\ge 0 \,|\; S_{nK} \cap A\neq \emptyset \big\},
\\
T(A)
& := \cup_{n=0}^{n_0} \big\{ [vu] \,|\; u \in S_{nK} \cap A \big\}.
\end{aligned}
$$
Let $T^1(A),\dots,T^s(A)$ be the connected components of the closure of $G(A) \setminus T(A)$.
Thus, $T^i(A)$ is a tree with $|T^i(A)\cap T(A)|=1$ for each $1 \le i \le s$.
Furthermore, if $T^i(A)\cap T(A)=\{a^i\}$, then $a^i \neq a^j$ for $i \neq j$ and $\max \{ d(a,a^i) \,|\; a \in V(T^i(A)) \} < K$.
Hence, Lemma \ref{l4:caract_complete} gives $|V(T^i(A))|\leq K|\p_e V(T^i(A))|$.
Since $\p_e V(T^i(A)) \subset \p_e A$ and $\p_e V(T^i(A)) \cap \p_e V(T^j(A)) = \emptyset$ for $i \neq j$, we have
$$
\sum_{i=1}^s |V(T^i(A))|
\leq K \sum_{i=1}^s |\p_e V(T^i(A))|
\leq K |\p_e A|.
$$
We are going to bound $|V(T(A))|$.
If $n_0\ge 1$, then
$$
|V(T(A))|
= 1 + \sum_{n=1}^{n_0} \sum_{j=1}^{K} |V(T(A))\cap S_{(n-1)K+j}|
\le 1+ K \sum_{n=1}^{n_0} |V(T(A))\cap S_{nK}|.
$$
Let us define a tree $T'$ with vertices $\cup_{n\in \mathbb{N}\cup \{0\}}S_{nK}$ and such that given $a_n\in S_{nK}, a_{n+1}\in S_{(n+1)K}$, $a_na_{n+1}$ defines an edge if and only if
$a_{n+1}\in T^v_{a_n}$.
Thus, $(T',v)$ is a $1$-pseudo-regular rooted tree.
If we define the set $A'\subset V(T')$ as $A'= A\cap V(T')$, then
$$
|V(T(A))|
\le K \Big( 1+ \sum_{n=1}^{n_0} |V(T(A))\cap S_{nK}| \Big)
= K \Big( 1+ \sum_{n=1}^{n_0} |A'\cap S_{nK}| \Big)
= K |A'| .
$$
Lemma \ref{l2:caract_complete} gives $|A'| \le 6|\p_{e,T'}^1 A'|$.
We define a map $F:\p_{e,T'}^1 A' \rightarrow \p_e A$ as follows.
For each $a \in \p_{e,T'}^1 A'$ there exists $1\le n \le n_0$ with $a\in A\cap S_{nK}$
and $A\cap T_a^v \cap S_{(n+1)K}=\emptyset$.
Thus, we can choose $a' \in \p_e A \cap T_a^v$ and define $F(a)=a'$.
Since $F$ is an injective map, we have
$$
|V(T(A))|
\le K |A'|
\le 6K |\p_{e,T'}^1 A'|
\le 6K |\p_e A| .
$$
Note that $|V(T(A))| \le 6K |\p_e A|$ trivially holds if $n_0=0$, since in this case $|V(T(A))| = 1$ and $|\p_e A|\ge 1$.
Hence, we conclude in any case
$$
|A|
\le |V(T(A))|+\sum_{i=1}^s |V(T^i(A))|
\le 6K |\p_e A| + K |\p_e A|
= 7K |\p_e A|.
$$
The same argument proves that $|A|\le 7K |\p_e A|$ holds for every finite set $A\subset V(T)$ with $G(A)$ connected and $v\notin A$.

Fix a finite set $A\subset V(T)$.
We can write $A= A_1 \cup \cdots \cup A_r$, where $G(A_1), \dots , G(A_r)$ are the connected components of $G(A)$.
We have proved $|A_j| \le 7K |\p_e A_j|$ for $1 \le j \le r$.
Since $T$ is a tree, the sets $\{\p_e A_1, \dots , \p_e A_r\}$ are pairwise disjoint and $\p_e A_1 \cup \cdots \cup \p_e A_r \subseteq \p A$, and we conclude
$$
|A|
= \sum_{j=1}^{r} |A_j|
\le \sum_{j=1}^{r} 7K |\p_e A_j|
\le 7K |\p A|.
$$
Since $A$ is an arbitrary finite set in $V(T)$, we conclude $h(T)\geq \frac{1}{7K}$.

Let us assume now that $T$  is not pseudo-regular, this is, for every $K \in \NN$, there exists a vertex
$a_K\in V(T)$ such that $T_{a_K}^v \cap S(v,d(v,a_K)+K)$ has just one vertex.
Therefore, since $(T,v)$ is a geodesically complete rooted tree, $A^K:=T_{a_K}^v \cap B(v,d(v,a_K)+K)$ satisfies $|A^K| = K$ and $|\partial A^K| =2$.
Thus, $\frac{|\partial A^K|}{|A^K|} = \frac{2}{K}$  for every $K \in \NN$ and we conclude $h(T)=0$.
\end{proof}

\begin{remark} Let $T_k$ be the homogeneous $k$-tree with $k\geq 3$ and fix any vertex $v$ in $T_k$. It is immediate to see that $(T_k,v)$ is $1$-pseudo-regular.
Therefore, by Proposition \ref{p:caract_complete}, $h(T_k) \ge \frac17 >0$.
\end{remark}

As usual, we say a graph $\G$ is \emph{infinite} if $| V(\G) | = \infty$, and that it is \emph{unbounded} if $\diam \G = \infty$.
If $\G$ is a graph with $|N(v)|<\infty$ for every $v\in V(\G)$ (in particular, if $\G$ is uniform), then $\G$ is unbounded if and only if it is infinite.

Let us consider an unbounded rooted tree $(T,v)$, its maximal geodesically complete subtree $(T_\infty,v)$ and $C\in \NN$.
We say that $(T,v)$ is $C$-\emph{complemented} if $|V(T_0)| \le C$ for every connected component $T_0$ of the closure of $T \setminus T_\infty$.
$(T,v)$ is \emph{complemented} if it is $C$-complemented for some $C\in \NN$.
Notice that this implies that $v$ is a pole.
In fact, for uniform trees, being complemented is equivalent to having a pole.

\smallskip

The previous results allow to prove the main result in this section,
which characterizes the trees verifying the isoperimetric inequality.

\begin{theorem} \label{t:iitreefinal}
Let $(T,v)$ be an unbounded rooted tree.
Then $h(T)>0$ if and only if $(T_\infty,v)$ is pseudo-regular and $(T,v)$ is complemented.
Furthermore, if $(T_\infty,v)$ is $K$-pseudo-regular and $(T,v)$ is $C$-complemented, with $K,C\ge 1$, then
$$
h(T_\infty)
\ge h(T)
\ge \frac{h(T_\infty)}{C + (C-1)h(T_\infty)}
\ge \frac{1}{(7K+1)C -1}\,.
$$
\end{theorem}

\begin{proof}
Let us prove first $h(T_\infty) \ge h(T)$.
Fix a finite set $A \subset V(T_\infty)$.
Let $T_1,\dots,T_r$ be the connected components of the closure of $T\setminus T_\infty$ with $T_j\cap A \neq \emptyset$.
Define $A'=\cup_{j=1}^r V(T_j) \cup A \subset V(T)$.
Thus,
$$
|A|
\le |A'|
\le h(T)^{-1} |\p A'|
= h(T)^{-1} |\p_{T_\infty} A|,
$$
and $h(T_\infty) \ge h(T)$.

Assume now that $(T_\infty,v)$ is $K$-pseudo-regular and $(T,v)$ is $C$-complemented.
Let $\{T_j\}_{j\in J}$ be the connected components of the closure of $T\setminus T_\infty$
and $\{v_j\}= T_j \cap T_\infty$ for each $j\in J$.
Note that $\{T_j\}_{j\in J}$ are pairwise disjoint.
Fix a finite set $A \subset V(T)$ and define
$$
\begin{aligned}
A_j
= A\cap \big( T_j \setminus \{v_j\} \big),
\qquad
& J_0
= \big\{ j\in J \,|\; A_j \neq \emptyset \big\},
\\
J_1
= \big\{ j\in J_0 \,|\; \p A \cap T_j \neq \emptyset \big\},
\qquad
& J_2
= J_0 \setminus J_1
= \big\{ j\in J_0 \,|\; V(T_j) \subseteq A \big\}.
\end{aligned}
$$
If $j\in J_1$, then
$$
\sum_{j\in J_1} |A_j|
\le \sum_{j\in J_1} (C-1)
\le \sum_{j\in J_1} (C-1) |\p A \cap T_j|
\le (C-1) |\p A|.
$$
Furthermore,
$$
\sum_{j\in J_2} |A_j|
\le \sum_{j\in J_2} (C-1)
= (C-1) \sum_{j\in J_2}|\{v_j\}|
\le (C-1) |A \cap T_\infty|.
$$
Since $h(T_\infty) \ge \frac{1}{7K}>0$ by Proposition \ref{p:caract_complete}, we have
$$
|A \cap T_\infty|
\le h(T_\infty)^{-1}|\p_{T_\infty}(A \cap T_\infty)|
\le h(T_\infty)^{-1}|\p A|.
$$
We conclude
$$
\begin{aligned}
|A|
& = |A \cap T_\infty| + \sum_{j\in J_1} |A_j| + \sum_{j\in J_2} |A_j|
\\
& \le |A \cap T_\infty| + (C-1)|\p A| + (C-1)|A \cap T_\infty|
\\
& \le \big(C \, h(T_\infty)^{-1}+C-1\big) |\p A|,
\end{aligned}
$$
and
$$
h(T)
\ge \frac{h(T_\infty)}{C + (C-1)h(T_\infty)}
\ge \frac{1}{(7K+1)C -1}
>0.
$$

Finally, assume that $h(T)> 0$.
Hence, $h(T_\infty) \ge h(T) > 0$ and Proposition \ref{p:caract_complete} gives that $(T_\infty,v)$ is pseudo-regular.
Since $(T_\infty,v)$ is the maximal geodesically complete subtree of $(T,v)$, any connected component of the closure of $T\setminus T_\infty$ is a finite tree.
Seeking for a contradiction, assume that $(T,v)$ is not complemented.
Given any $C\in \NN$, there exists a connected component $T_C$ of the closure of $T\setminus T_\infty$ with $|V(T_C)| > C$.
Let $v_C$ be the vertex with $\{v_C\} = T_C \cap T_\infty$ and $A_C=V(T_C) \setminus \{v_C\}$. Thus,
$$
\frac{|\p A_C|}{|A_C|}
=\frac{|\{v_C\}|}{|V(T_C) \setminus \{v_C\}|}
\le \frac{1}{C} \,.
$$
Since this inequality holds for every $C\in\NN$, we have $h(T) = 0$, which is a contradiction.
Hence, $(T,v)$ is complemented.
\end{proof}

Note that, since the condition $h(T)> 0$ is independent of $v$, we can choose any root $v$.

\smallskip

In the case of uniform unbounded trees, the characterization can be given in terms of the boundary at infinity.

\smallskip

Let us adapt the following definition from \cite{BS} where we introduce the constant $\varepsilon_0$ for convenience. Notice that for bounded metric spaces both definitions coincide. Since herein this property will be always applied to compact spaces all the results work as well with the original definition.

\begin{definition} Given a metric space $(X,d)$ and a constant $S>1$, we say that $(X,d)$ is \emph{$S$-uniformly perfect} if there exists some $\varepsilon_0>0$ such that for every $x\in X$ and every $0<\varepsilon \le \varepsilon_0$ there exist a point $y\in X$ such that
$\frac{\varepsilon}{S} <d(x,y) \le \varepsilon$. We say that $(X,d)$ is \emph{uniformly perfect} if there exists some $S$ such that $(X,d)$ is $S$-uniformly perfect.
\end{definition}

\begin{lemma} \label{lema: pseudo-regular} A metric space $(X,d)$ is uniformly perfect if and only if there is a constant $R>1$ and some $\varepsilon_0>0$ such that for every $x\in X$ and every $0<\varepsilon \le \varepsilon_0$ there exist at least two points $y_1,y_2\in X$ such that $d(x,y_i) \le \varepsilon$ for $i=1,2$ and $\frac{\varepsilon}{R} < d(y_1,y_2) $.
\end{lemma}

\begin{proof} The if part is immediate with $S=2R$ and either $y=y_1$ or $y=y_2$.

Suppose that $(X,d)$ is $S$-uniformly perfect and consider $0<\varepsilon\le \varepsilon_0$. Then, there exist $y_1\in X$ such that $\frac{\varepsilon}{S} < d(x,y_1) \leq \varepsilon$ and
$y_2\in X$ such that $\frac{\varepsilon}{S^3} < d(x,y_2) \leq \frac{\varepsilon}{S^2}$. Therefore,
$d(y_1,y_2)\geq d(x,y_1)-d(x,y_2) >\frac{\varepsilon}{S}-\frac{\varepsilon}{S^2}= \frac{S-1}{S^2}\, \varepsilon.$ Thus, it suffices to consider $R=\frac{S^2}{S-1}$.
\end{proof}

\begin{remark}\label{r:e1}
It is easy to check that if $X$ is uniformly perfect with constants $S$ and $\varepsilon_0$, and $\varepsilon_0'$ is any positive constant,
then $X$ is also uniformly perfect with constants $S\max\big\{1,\frac{\varepsilon_0'}{\varepsilon_0}\big\}$ and $\varepsilon_0'$.
In a similar way, if the equivalent condition to uniform perfectness in
Lemma \ref{lema: pseudo-regular} holds with constants $R$ and $\varepsilon_0$,
then this condition holds with constants $R\,\max\big\{1,\frac{\varepsilon_0'}{\varepsilon_0}\big\}$ and $\varepsilon_0'$.
\end{remark}

\begin{proposition}\label{p:equiv} A geodesically complete rooted tree $(T,v)$
is pseudo-regular if and only if $end(T,v)$ is uniformly perfect.
\end{proposition}

\begin{proof} Suppose $(T,v)$ is $K$-pseudo-regular. Given any point $F\in end(T,v)$ and any $\varepsilon>0$,
let $a=F(\lceil-\ln(\varepsilon)\rceil)$, where $\lceil x \rceil$ denotes the upper integer part of $x$, i.e., the smallest integer greater or equal than $x$.
Then, there exist $y_1,y_2\in S(v,\lceil-\ln(\varepsilon)\rceil+K)\cap T^v_a$.  Since $T$ is
geodesically complete, there exist $G_1,G_2\in end(T,v)$ with $G_i(\lceil-\ln(\varepsilon)\rceil+K)=y_i$ and at least one of them, let us assume $G_1$, satisfies that $G_1(\lceil-\ln(\varepsilon)\rceil+K) \neq F(\lceil-\ln(\varepsilon)\rceil+K)$. Since $y_1\in T^v_a$, thus
$(F|G_1)_v \ge \lceil-\ln(\varepsilon)\rceil \ge-\ln(\varepsilon)$ and $d_v(F,G_1)\leq \varepsilon$.
Since $G_1(\lceil-\ln(\varepsilon)\rceil+K) \neq F(\lceil-\ln(\varepsilon)\rceil+K)$, then $(F|G_1)_v \le \lceil-\ln(\varepsilon)\rceil + K-1 <-\ln(\varepsilon)+K$ and
$\frac{\varepsilon}{e^K} = e^{\ln(\varepsilon)-K} < d_v(F,G_1)$.
Hence, $end(T,v)$ is $e^K$-uniformly perfect.

Suppose that $end(T,v)$ is uniformly perfect and suppose $R>1$ satisfying the condition in Lemma \ref{lema: pseudo-regular}.
Consider any $a\in V(T)$.
Since $T$ is geodesically complete, there exist some $F\in  end(T,v)$ and some $t_0\geq 0$ such that $F(t_0)=a$. Thus, $a \in S(v,t_0)$. Let $\varepsilon:=e^{-t_0}$.
Then, there exist at least two points $G_1,G_2\in end(T,v)$ such that $d_v(F,G_i) \le \varepsilon$ for $i=1,2$ and $\frac{\varepsilon}{R} < d_v(G_1,G_2) $.
Since $d_v(F,G_i) \le \varepsilon$, $G_i(t)\in T_a^v$ for every $t\geq t_0$ and $i=1,2$. Also, since $d_0:=\frac{\varepsilon}{R} < d_v(G_1,G_2) $, $G_1(-\ln(d_0))=G_1(t_0+\ln(R))\neq G_2(t_0+\ln(R))=G_2(-\ln(d_0))$.
If we define $y_i=G_i(t_0+ \lceil \ln(R) \rceil )$ for $i=1,2,$ then $y_1 \neq y_2$.
Therefore, $(T,v)$ is $\lceil \ln(R) \rceil$-pseudo-regular.
\end{proof}

Theorem \ref{t:iitreefinal} and Proposition \ref{p:equiv} allow to obtain the following characterization of uniform trees with isoperimetric inequality.

\begin{theorem} \label{t:iitreeuni}
Let $(T,v)$ be a uniform infinite rooted tree.
Then $h(T)>0$ if and only if $v$ is a pole of $T$ and $end(T,v)$ is uniformly perfect.
\end{theorem}

\section{Hyperbolic graphs}

Let us recall the concepts of geodesic and sequential boundary of a hyperbolic space and some basic properties. For further information and proofs we refer the reader to \cite{BH,BS,GH,G1}.

\smallskip

Let $X$ be a hyperbolic space and $o\in X$ a base point.

The \emph{relative geodesic boundary} of $X$ with respect to the base-point $o$ is
\[\partial_o^g X := \{ [\gamma]  \, | \, \gamma: [0,\infty) \to X \mbox{ is a geodesic ray with } \gamma(0) = o\},\]
where  $\gamma_1 \sim \gamma_2$ if there exists some $K>0$ such that
$d(\gamma_1(t),\gamma_2(t))<K$, for every $t\geq 0.$

In fact, the definition above is independent from the base point.
Therefore, the set of classes of geodesic rays is called \emph{geodesic boundary} of $X$, $\partial^gX$.  Herein, we do not distinguish between the geodesic ray and its image.

A sequence of points $\{x_i\}\subset X$ \emph{converges to infinity}
if \[\lim_{i,j\to \infty} (x_i|x_j)_o=\infty.\] This property is
independent of the choice of $o$ since
\[|(x|x')_o-(x|x')_{o'}|\leq d(o,o')\] for any $x,x',o,o' \in X$.

Two sequences $\{x_i\},\{x'_i\}$ that converge to infinity are
\emph{equivalent} if \[\lim_{i\to \infty} (x_i|x'_i)_o=\infty.\]
Using the $\delta$-inequality, we easily see that this defines an
equivalence relation for sequences in $X$ converging to infinity.
The \emph{sequential boundary at infinity} $\partial_\infty X$ of $X$ is
defined to be the set of equivalence classes of sequences
converging to infinity.

Note that given a geodesic ray $\gamma$, the sequence $\{\gamma(n)\}$ converges to infinity and two equivalent rays induce equivalent sequences.
Thus, in general, $\partial^g X\subseteq \partial_\infty X$.

We say that a metric space is \emph{proper} if every closed ball is compact.
Every uniform graph and every complete Riemannian manifold are proper geodesic spaces.

\begin{proposition}\cite[Chapter III.H, Proposition 3.1]{BH}\label{Prop: equiv_boundary} If $X$ is a proper hyperbolic geodesic space, then the natural map from $\partial^g X$ to $\partial_\infty X$ is a bijection.
\end{proposition}

For every $\xi, \xi' \in \partial_\infty X$, its Gromov product
with respect to the base point
$o\in X$ is defined as
\[ (\xi|\xi')_o =  \inf \ \liminf_{i\to \infty} (x_i|x'_i)_o,\]
where the infimum is taken over all sequences $\{x_i\} \in \xi $, $\{x'_i\} \in \xi' $.

\begin{remark}\cite[Lemma 2.2.2]{BS}\label{Remark: geodesic_product}
If $X$ is a $\d$-hyperbolic geodesic space, then for every pair of geodesic rays
$\sigma, \sigma'$ with $\sigma(0)=x_0=\sigma'(0)$
and such that $ \{\sigma(n)\}\in \xi$ and $\{\sigma'(n)\} \in \xi'$,
\[(\xi|\xi')_{x_0}\leq \liminf_{t\to \infty} (\sigma(t)|\sigma'(t))_{x_0}\leq \limsup_{t\to \infty} (\sigma(t)|\sigma'(t))_{x_0}\leq (\xi|\xi')_{x_0}+2\delta.\]
\end{remark}

The Gromov product
\[(x|\xi)_o = \inf \ \liminf_{i\to \infty} (x|x_i)_o\]
is defined for any $x \in X$, $\xi \in \partial_\infty X$, where the infimum is taken over all
sequences $\{x_i\} \in \xi$.

A hyperbolic space $X$ is said to be \emph{visual}, if for some base point
$o\in X$ there is some constant $D>0$ such that for every $x\in X$ there
is $\xi \in \partial_\infty X$ with $d(o,x) \leq  (x|\xi )_o + D$. Moreover, this property
is independent of the choice of $o$.

\smallskip

Let us recall the following from \cite[p. 408]{BH}:

As usual, the \emph{degree} of a vertex $v\in V(\G)$ is defined as $\deg (v)=|N(v)|$.
Given three points $x,y,z$, in a metric space, there exist unique non-negative numbers $a,b,c$ such that $d(x,y)=a+b$, $d(x,z)=a+c$ and
$d(y,z)=b+c$. Let $T(a,b,c)$ be the metric tree that has three vertices of degree one, one vertex, $o_T$, of degree three and edges of length $a$, $b$ and $c$.

Given a geodesic triangle $\Delta=\{x,y,z\}$ in a geodesic space, there is a map $\Psi_\Delta: \Delta \to T(a,b,c)$ whose restriction to each side of $\Delta$ is an isometry. Let $o_{\Delta}=\Psi_\Delta^{-1}(o_T)$.

\begin{proposition}\cite[Chapter III.H, Proposition 1.17]{BH}\label{Prop: insize} Given a geodesic space $X$, the following conditions are equivalent:
\begin{itemize}
	\item[(1)] $X$ is $\delta$-hyperbolic,
	\item[(2)] there exists $\delta_1>0$ such that if $p,q\in \Psi_\Delta^{-1}(t)$, then $d(p,q)<\delta_1$ for every geodesic triangle $\Delta$ in $X$,
	\item[(3)] there exists $\delta_2>0$ such that $\diam(o_\Delta)<\delta_2$ for every geodesic triangle $\Delta$ in $X$.
\end{itemize}
Furthermore, if $(1)$ holds, then $\delta_1$ and $\delta_2$ are constants which just depend on $\d$.
\end{proposition}

\begin{proposition} \label{Prop:polevisual}
A proper hyperbolic geodesic space has a pole if and only if it is visual.
\end{proposition}

\begin{proof} Suppose that there exist $v\in X$ and $M>0$ such that for any $x\in X$ there is a geodesic ray
$\sigma$ with $\sigma(0)=v$ and $d(x,\sigma)<M$.

Let $t_0\geq 0$ such that $d(x,\sigma(t_0))<M$. Hence, $t_0-M < d(v,x)<t_0+M$.

Also, $\{\sigma(i)\}\in \xi$ for some $\xi\in \partial_\infty X$. Then, for every  $\{x_i\}\in \xi$,
$\lim_{i\to \infty}(x_i|\sigma(i))_v=\infty$.
Since $X$ is $\delta$-hyperbolic for some constant $\d\ge 0$, then
\[(x|x_i)_v\geq \min\{(\sigma(i)|x_i)_v,(x|\sigma(i))_v\} -\delta=(x|\sigma(i))_v-\delta\]
for $i$ big enough.
Since $d(v,x) > t_0-M$, $d(v,\sigma(i))=i$ and $d(x,\sigma(i))\leq i-t_0+M$, we have
$$
\begin{aligned}
(x|\xi)_v
&\geq \liminf_{i\to \infty} (x|\sigma(i))_v-\delta
= \liminf_{i\to \infty} \frac{1}{2} \big(d(v,x)+d(v,\sigma(i))-d(x,\sigma(i))\big) -\delta
\\
& \geq \liminf_{i\to \infty} \frac{1}{2}(t_0-M+i-i+t_0-M) -\delta = t_0-M-\delta.
\end{aligned}
$$

Thus, $d(v,x)<t_0+M \leq (x|\xi)_v +2M +\delta$.

Now, suppose that there exist $v\in X$ and $D>0$ such that for every $x\in X$ there is some $\xi \in \partial_\infty X$ with $d(v,x)\leq (x|\xi)_v + D$.
By Proposition \ref{Prop: equiv_boundary}, there is some $\sigma\in \partial^g X$ such that $\sigma(0)=v$ and $\{\sigma(i)\}\in \xi$. Then,
\[
d(v,x)-D\leq (x|\xi)_v\leq \liminf_{i\to \infty} (x|\sigma(i))_v=\liminf_{i\to \infty} \frac{1}{2}\big(d(v,x)+d(v,\sigma(i))-d(x,\sigma(i))\big).
\]

Let us consider the geodesic triangle $\Delta=\{v,x,\sigma(i)\}$ and let $p=[vx]\cap o_\Delta$, $q=[x\sigma(i)]\cap o_\Delta$ and $\sigma(t_0)=[v\sigma(i)]\cap o_\Delta$.
By Proposition \ref{Prop: insize}, there exists $\delta_2=\d_2(\d)$ such that $\diam \{p,q,\sigma(t_0)\}\leq \delta_2$.

Hence, since $d(v,x)\leq d(v,\sigma(t_0))+d(\sigma(t_0),x)$ and
$$
d(x,\sigma(i))
=d(x,q)+d(q,\sigma(i))
=d(x,q)+d(\sigma(t_0),\sigma(i))
\geq d(x,\sigma(t_0))+d(\sigma(t_0),\sigma(i))-\delta_2,
$$
we have for any $\varepsilon >0$ and $i$ big enough
$$
\begin{aligned}
2d(v,x)-2D
& \leq d(v,\sigma(t_0))+d(\sigma(t_0),x) +i-d(x,\sigma(t_0))-d(\sigma(t_0),\sigma(i))+\delta_2 +2\varepsilon
\\
&=t_0+d(\sigma(t_0),x) +i-d(x,\sigma(t_0))-i+t_0+\delta_2 +2\varepsilon
\le 2t_0+2\delta_2+2\varepsilon.
\end{aligned}
$$

Also,
\[
d(v,x)
=d(v,p)+d(p,x)
=d(v,\sigma(t_0))+d(p,x)
\geq d(v,\sigma(t_0))+d(\sigma(t_0),x)-\delta_2
\]
and, therefore,
\[d(\sigma(t_0),x)\leq d(v,x)- d(v,\sigma(t_0))+\delta_2\leq t_0+\delta_2 +\varepsilon +D-t_0+\delta_2= D+2\delta_2+\varepsilon\]
and $X$ has a pole in $v$.
\end{proof}

A metric $d$ on the sequential boundary at infinity
$\partial_\infty X$ of $X$ is said to be \emph{visual}, if there are $o\in X$, $a > 1$ and positive
constants $c_1$, $c_2$, such that
\[c_1a^{-(\xi|\xi')_o} \leq  d(\xi, \xi') \leq c_2a^{-(\xi|\xi')_o}\]
for all $\xi, \xi' \in \partial_\infty X$. In this case, we say that $d$ is a visual metric
with respect to the base point $o$ and the parameter $a$.

\begin{theorem}\cite[Theorem 2.2.7]{BS} Let $X$ be a hyperbolic space. Then for any $o \in X$,
there is $a_0 > 1$ such that for every $a \in (1, a_0]$ there exists a metric $d$ on
$\partial_\infty X$, which is visual with respect to $o$ and $a$.
\end{theorem}

\begin{remark} Notice that for any visual metric, $\partial_\infty X$ is bounded and complete.
\end{remark}

Let us recall the following construction from \cite{BS}.

A subset $A$ in a metric space $(X,d)$ is called \emph{$r$-separated},
$r>0$, if $d(a,a')\geq r$ for any distinct $a,a'\in A$. Note that
if $A$ is maximal with this property, then the union $\cup_{a\in
A} B_r(a)$ covers $X$. A maximal $r$-separated set $A$ in a metric
space $X$ is called an $r$-\emph{approximation} of $X$.

A \emph{hyperbolic approximation} of a metric space $X$, $\mathcal{H}(X)$, is a graph defined as follows.
Fix a positive $r\leq \frac{1}{6}$ which is called the \emph{parameter} of $\mathcal{H}(X)$.
For every $k\in \mathbb{Z}$, let $A_k\in X$ be an
$r^k$-approximation of $X$. For every $a\in A_k$, consider the ball
$B(a,2r^k)\subset X$.
Let us define, for every $k$, the set $V^*_k:=\{B(a,2r^k)\,|\;a\in A_k\}$ and a set $V_k$ which has a
vertex corresponding to each ball in $V^*_k$. Then let $V=\cup_{k\in \mathbb{Z}}V_k$ be the set
of vertices of the graph $\mathcal{H}(X)$.
Thus, every vertex $v \in V$ corresponds
to some ball $B(a, 2r^k)$ with $a \in A_k$ for some $k$. Let us denote the corresponding ball to $v\in V$ simply by $B(v)$.

There is a natural level function $l:V \to \mathbb{Z}$ defined by $l(v)=k$
for $v\in V_k$.

Vertices $v,v'$ are connected by an edge if and only if they
either belong to the same level, $V_k$, and the closed balls
$\bar{B}(v),\bar{B}(v')$ intersect, $\bar{B}(v)\cap
\bar{B}(v')\neq \emptyset$, or they lie on neighboring levels
$V_k,V_{k+1}$ and the ball of the upper level, $V_{k+1}$, is
contained in the ball of the lower level, $V_k$.

Since $A_k$ is an $r^k$-approximation of $X$ for any $k\in\ZZ$ and $r\leq \frac{1}{6}$, every vertex in $V_k$ has a neighbor in $V_{k+1}$.

An edge $vv'\subset \mathcal{H}(X)$ is called \emph{horizontal} if its
vertices belong to the same level, $v,v'\in V_k$ for some $k\in
\mathbb{Z}$. Other edges are called \emph{radial}. Consider the
path metric on $X$ for which every edge has length 1.

Note that any (finite or infinite) sequence $\{v_k\}\in V$ such
that $v_kv_{k+1}$ is a radial edge for every $k$ and such that the level
function $l$ is monotone along $\{v_k\}$, is the vertex sequence
of a geodesic in $\mathcal{H}(X)$. Such a geodesic is called \emph{radial}.

\begin{proposition}\cite[Proposition 6.2.10]{BS}\label{Prop: hyp} For any metric space $(X,d)$, $\mathcal{H}(X)$ is a
geodesic $3$-hyperbolic space.
\end{proposition}

Assume now that $X$ is bounded and non-trivial.
Then, since $r<1$, there is a maximal integer $k$ with $\diam X<r^k$ and it is
denoted by $k_0=k_0(\diam X,r)$. Then, for every $k\leq k_0$ the vertex
set $V_k$ consists of one point, and therefore contains no
essential information about $X$. Thus, the graph $\mathcal{H}(X)$ can be modified
making $V_k=\emptyset$ for every $k<k_0$.
This modified graph is called the \emph{truncated hyperbolic approximation} of $X$, $\mathcal{H}^t(X)$. The level function $l$ has a unique minimum, $v$, with $l(v)=k_0$. This point $v$ can be considered as the natural base point of the truncated hyperbolic approximation.

\begin{theorem}\cite[Proposition 6.4.1]{BS} \label{t:equalboundaries}
Let $\Gamma$ be a truncated hyperbolic approximation of a
complete bounded metric space $(X,d)$. Then there is a canonical identification
$\partial_\infty \Gamma = X$ under which the metric $d$ of $X$ is a visual metric on
$\partial_\infty \Gamma$ with respect to the natural base point $v$ of $\Gamma$ and the parameter $a = \frac{1}{r}$.
\end{theorem}

The following is inspired by the definition of a metric space having bounded geometry in \cite{BDHM}.

\begin{definition} Given a metric space $(X,d)$, we say that $X$ has \emph{strongly bounded geometry} if for every $K>0$ there exists $M>0$ such that the following condition is satisfied:
for any $\varepsilon >0$, any $\varepsilon$-approximation of $X$, $A_\varepsilon$, and any $x\in X$, $|A_\varepsilon \cap B(x,K\varepsilon)|<M$.
\end{definition}

\begin{remark} \label{rem:diam}
Note that if $\varepsilon> \diam X$, then $|A_\varepsilon|=1$ and $|A_\varepsilon \cap B(x,K\varepsilon)| \le |A_\varepsilon|=1$.
Hence, in order to check whether $X$ has strongly bounded geometry or not, it suffices to consider $0 < \varepsilon \le \diam X$.
\end{remark}

\begin{proposition}\label{Prop:uniform} If the metric space $(X,d)$ has strongly bounded geometry, then $\mathcal{H}(X)$ is uniform.
\end{proposition}

\begin{proof} Let $v$ be any vertex in $\mathcal{H}(X)$.
Let $r$ be the parameter of $\mathcal{H}(X)$ and suppose that $B(v)=B(x,2r^n)$ for some $x\in X$ and $n\in \mathbb{N}$.

For each $k\in \mathbb{N}$, let $A_k$ be the $r^k$-approximation from the construction of $\mathcal{H}(X)$. Since $X$ has strongly bounded geometry, there is some constant $M_1$ such that $|A_n \cap B(x,5r^n)|<M_1$.
Thus, there are less than $M_1$ vertices adjacent to $v$ at level $n$.

Also, there is a constant $M_2$ such that $|A_{n+1} \cap B(x,\frac{2}{r}\cdot r^{n+1})|<M_2$. Therefore, there are less than $M_2$ vertices adjacent to $v$ at level $n+1$.

Finally, there is a constant $M_3$ such that $|A_{n-1} \cap B(x,2\cdot r^{n-1})|<M_3$. Thus, there are less than $M_3$ vertices adjacent to $v$ at level $n-1$.

Hence, $|N(v)|<M_1+M_2+M_3$. Since $M_1$, $M_2$ and $M_3$ only depend on a fixed parameter $r$, then $|N(v)|<M_1+M_2+M_3$ for every vertex $v$ in $\mathcal{H}(X)$. Therefore, $\mathcal{H}(X)$ is uniform.
\end{proof}

A map $f: X \to X'$ between metric spaces is said to be a \emph{rough similarity}
if there exist constants $a>0$ and $b,\varepsilon \ge0$ such that
$a\, d(x,x')-b \leq d'(f(x),f(x'))\leq a\, d(x,x')+b$ for all $x, x' \in X$ and $f$ is $\varepsilon$-\emph{full}.
If there is a rough similarity between $X$ and $X'$, then the spaces $X$ and $X'$ are called \emph{roughly similar} to each other.

\begin{theorem}\cite[Corollary 7.1.6]{BS} \label{th:quasi-isometric} Visual hyperbolic geodesic spaces $X,X'$ with
bilipschitz equivalent boundaries at infinity are roughly similar
to each other. In particular, every visual hyperbolic space is
roughly similar (and therefore, quasi-isometric) to any (truncated) hyperbolic approximation of its boundary at
infinity; and any two hyperbolic approximations of a complete
bounded metric space $Z$ are roughly similar (and therefore, quasi-isometric) to each other.
\end{theorem}

Theorem \ref{th:quasi-isometric}, Propositions \ref{Prop:polevisual} and \ref{Prop:uniform} and Lemma \ref{l:quasi-isometric} immediately yield the following.

\begin{proposition}\label{Prop:hyperbolic-equiv}
Let $\Gamma$ be a hyperbolic uniform graph with a pole, and such that $\partial_\infty \Gamma$ has strongly bounded geometry for some visual metric.
Then $h(\Gamma)>0$ if and only if $h(\mathcal{H}^t(\partial_\infty \Gamma))>0$.
\end{proposition}

\begin{proposition}\label{Prop: bounded} If $\Gamma$ is a hyperbolic uniform graph with a pole, then $\partial_\infty \Gamma$ with any visual metric has strongly bounded geometry.
\end{proposition}

\begin{proof} Let $x_0$ be a pole of $\G$ and $d$ any metric in $\partial_\infty \Gamma$ with
\[c_1a^{-(\xi|\xi')_{x_0}}\leq d(\xi,\xi') \leq c_2a^{-(\xi|\xi')_{x_0}}\]
for all $\xi,\xi'\in \partial_\infty \Gamma$.

Since $\G$ is $\d$-hyperbolic for some constant $\d\ge 0$,
Remark \ref{Remark: geodesic_product} gives that
for every pair of geodesic rays $\sigma, \sigma'$ with $\sigma(0)=x_0=\sigma'(0)$ such that $\{\sigma(n)\}\in \xi$ and $\{\sigma'(n)\}\in \xi'$,
\[\liminf_{t\to \infty} c_1a^{-(\sigma(t)|\sigma'(t))_{x_0}}\leq d(\xi,\xi') \leq
\liminf_{t\to \infty} c_2a^{2\delta} a^{-(\sigma(t)|\sigma'(t))_{x_0}}.\]

Consider any $K>0$ and any $\varepsilon >0$.
By Remark \ref{rem:diam}, we can assume that $0<\varepsilon \le \diam \partial_\infty \Gamma$.
Let $\sigma_0,\sigma$ be any two geodesic rays with starting in $x_0$.
If $[\sigma]\in B([\sigma_0],K\varepsilon)$, then
$$
\begin{aligned}
\liminf_{t\to \infty} c_1a^{-(\sigma(t)|\sigma_0(t))_{x_0}}
& < K\varepsilon,
\\
\limsup_{t\to \infty} (\sigma(t)|\sigma_0(t))_{x_0}
& >-\log_a{\frac{K\varepsilon}{c_1}}.
\end{aligned}
$$
By Remark \ref{Remark: geodesic_product},
$$
\liminf_{t\to \infty} (\sigma(t)|\sigma_0(t))_{x_0}
 >-\log_a{\frac{K\varepsilon}{c_1}}-2\d.
$$

Assume first that $t_1:=\big\lfloor-\log_a{\frac{K\varepsilon}{c_1}}-2\d\big\rfloor \geq 0$, where $\lfloor s \rfloor$ denotes the lower integer part of $s$, i.e., the greatest integer less or equal than $s$.
Thus, there exists $N$ big enough so that
$$(\sigma(N)|\sigma_0(N))_{x_0}=N-\frac{1}{2}d(\sigma(N),\sigma_0(N))>t_1\geq 0.$$
Let $c=\frac{1}{2}d(\sigma(N),\sigma_0(N))>0$.
Then, $d(x_0,\sigma(N))=d(x_0,\sigma_0(N))=(N-c)+c$ and $d(\sigma(N),\sigma_0(N))=2c$.
Therefore, considering the geodesic triangle $\D=\{x_0,\sigma(N),\sigma_0(N)\}$, by Proposition \ref{Prop: insize}, there exists $\delta_1=\d_1(\d)>0$ such that $d(\sigma(t),\sigma_0(t))<\delta_1$ for every $t\leq N-c$.
In particular, $d(\sigma_0(t_1),\sigma(t_1))<\delta_1$.

Let $A_\varepsilon$ be any $\varepsilon$-approximation of $\partial_\infty\Gamma$. Let $\{\gamma_i\}_{i\in I}$ be a set of geodesic rays with $\gamma_i\not \sim \gamma_j$ for every $i\neq j$ and such that $\{[\gamma_i]\}_{i\in I}=A_{\varepsilon}\cap B([\sigma_0],K\varepsilon)$.

Since $d([\gamma_i],[\gamma_j]) \ge \varepsilon$, then
$$\liminf_{t\to \infty} c_3 a^{- (\gamma_i(t)|\gamma_j(t))_{x_0}}>\varepsilon, \quad \mbox{ where } c_3=2c_2a^{2\delta}.$$
Hence,
$$
\limsup_{t\to \infty} (\gamma_i(t)|\gamma_j(t))_{x_0} < -\log_a{\frac{\varepsilon}{c_3}}\,.
$$
Since the Gromov product is non-negative, $0 < -\log_a{\frac{\varepsilon}{c_3}}$
and $t_2:= \big\lceil-\log_a{\frac{\varepsilon}{c_3}}\big\rceil\geq 1$, where $\lceil s \rceil$ denotes the upper integer part of $s$
(thus, we can consider $\gamma_i(t_2)$).
Therefore, for $N$ big enough,
$$(\gamma_i(N)|\gamma_j(N))_{x_0}=\frac{1}{2}\big(2N-d(\gamma_i(N),\gamma_j(N))\big) <t_2.$$
Hence, $2(N-t_2)<d(\gamma_i(N),\gamma_j(N))$ and $\gamma_i(t_2)\neq \gamma_j(t_2)$ for every $i\neq j$
(if $\gamma_i(t_2) = \gamma_j(t_2)$ for some $i\neq j$,
then $d(\gamma_i(N),\gamma_j(N))\leq d(\gamma_i(N),\gamma_i(t_2))+ d(\gamma_i(t_2),\gamma_j(t_2))+ d(\gamma_j(t_2),\gamma_j(N))=2(N-t_2)$ for every $N \ge t_2$, a contradiction).

Since for every $i\in I$, $\gamma_i(t_1)$ must be a vertex in $B(\sigma_0(t_1),\delta_1)$ and $\Gamma$ is $\mu$-uniform for some constant $\mu\in\NN$,
then $\gamma_i(t_1)$ is contained in a set of at most $\mu^{\delta_1}$ possible vertices in $S(x_0,t_1)$.
Also, since $\Gamma$ is $\mu$-uniform, if $t_2 \ge t_1$, then for every vertex
$w\in S(x_0,t_1)$ there exist at most $\mu^{t_2-t_1}$ vertices $w'_k$ in $S(x_0,t_2)$ such that $w\in [x_0w'_k]$. Therefore,
$$
\begin{aligned}
|A_{\varepsilon}\cap B([\sigma_0],K\varepsilon)|
& = |I|\leq \mu^{\delta_1} \cdot \mu^{t_2-t_1}
= \mu^{\delta_1+\lceil-\log_a{\frac{\varepsilon}{c_3}}\rceil - \lfloor-\log_a{\frac{K\varepsilon}{c_1}}-2\d \rfloor}
\\
& \le \mu^{\delta_1+2-\log_a{\frac{\varepsilon}{c_3}} + \log_a{\frac{K\varepsilon}{c_1}} + 2\d}
= \mu^{2+\delta_1+2\d+\log_a{\frac{c_3K}{c_1}}},
\end{aligned}
$$
which only depends on $K$ and some fixed constants ($\d,\mu,a,c_1$ and $c_2$).
If $t_2 < t_1$, then a similar and simpler argument gives $|A_{\varepsilon}\cap B([\sigma_0],K\varepsilon)| = |I|\leq \mu^{\delta_1} $.

Finally, consider the case $t_1<0$.
Thus, $-\log_a{\frac{K\varepsilon}{c_1}}- 2\d<0$ and $-\log_a \varepsilon < 2\d + \log_a{\frac{K}{c_1}}$.
Since $\Gamma$ is $\mu$-uniform and $\gamma_i(t_2)\neq \gamma_j(t_2)$ for every $i\neq j$, we have
$$
|A_{\varepsilon}\cap B([\sigma_0],K\varepsilon)|
 = |I|\leq \mu^{t_2}
= \mu^{\lceil-\log_a{\frac{\varepsilon}{c_3}}\rceil}
\le \mu^{1-\log_a{\frac{\varepsilon}{c_3}}}
< \mu^{1+2\d +\log_a{\frac{c_3K}{c_1}}}.
$$
\end{proof}

\begin{theorem} \label{t:iigraph}
Given a hyperbolic uniform graph $\Gamma$ with a pole, then $h(\Gamma)>0$ if and only if $\partial_\infty \Gamma$ is uniformly perfect for some visual metric.
\end{theorem}

\begin{proof}
Let $v_0$ be a pole of $\G$.
Fix a visual metric $d$ on $\partial_\infty \Gamma$. Then, by Proposition \ref{Prop: bounded},
$(\partial_\infty \Gamma,d)$ has strongly bounded geometry.

Consider $\Gamma'=\mathcal{H}^t(\partial_\infty \Gamma)$ any truncated hyperbolic approximation with parameter $r$, $r^k$-approximations $A_k$ and corresponding vertices $V_k$.
By Proposition \ref{Prop:hyperbolic-equiv}, $h(\Gamma)>0$ if and only if $h(\Gamma')>0$.

Suppose $\partial_\infty \Gamma$ is uniformly perfect and let us prove, using Proposition \ref{p:sufficient}, that $h(\Gamma')>0$.

First, since $\partial_\infty \Gamma=\partial_\infty \Gamma'$ by Theorem \ref{t:equalboundaries}, this boundary with the fixed metric $d$ has strongly bounded geometry.
Therefore, let $M_1>0$ be such that for every $\varepsilon >0$, any $\varepsilon$-approximation of $\partial_\infty \Gamma'$, $A$, and any $x\in \partial_\infty \Gamma'$,
$|A \cap B(x,5\varepsilon)|< M_1$. In particular, $|A \cap \bar{B}(x,4\varepsilon)|< M_1$.

Let $v$ be any vertex in $\Gamma'$ and suppose that $B(v)=B(x,2r^k)$.
Since $\partial_\infty \Gamma=\partial_\infty \Gamma'$ is uniformly perfect, then, by Lemma \ref{lema: pseudo-regular} there exist two points $y_1,y_2$ in $\bar{B}(x,r^k)$ such that
$d(y_1,y_2)>\frac{1}{R}r^k$ for some $R>1$
(since $\Gamma'$ is a truncated hyperbolic approximation, by Remark \ref{r:e1}, without loss of generality we can assume that $r^m \le \varepsilon_0$ for every $m$ with $A_m\neq\emptyset$).

Now, in each ball $\bar{B}(y_{i_1},\frac{1}{3R}r^k)$ with $i_1\in \{1,2\}$, again by Lemma \ref{lema: pseudo-regular}, there exist $y_{i_1,1},y_{i_1,2}$  such that
$d(y_{i_1,1},y_{i_1,2})>\frac{1}{3R^2}r^k$. Also, since $d(y_1,y_2)>\frac{1}{R}r^k$ and $d(y_{i_1},y_{i_1,i_2})\le \frac{1}{3R}r^k$ for every $i_1,i_2 \in \{1,2\}$, then $d(y_{1,i_2},y_{2,j_2})>\frac{1}{3R}r^k$ for every $i_2,j_2 \in \{1,2\}$.

If $2^2>M_1$ we are done. Otherwise, we repeat the process. Thus, for every point $y_{i_1,i_2,\dots,i_{t-1}}$, in each ball $\bar{B}(y_{i_1,i_2,\dots,i_{t-1}},\frac{1}{3^{t-1}R^{t-1}}r^k)$ there exist two points $y_{i_1,i_2,\dots,i_{t-1},1}, y_{i_1,i_2,\dots,i_{t-1},2}$ such that
$$d(y_{i_1,i_2,\dots,i_{t-1},1}, y_{i_1,i_2,\dots,i_{t-1},2})> \frac{1}{3^{t-1}R^{t}}r^k$$
and
$$d(y_{i_1,i_2,\dots,i_{t-1},i_{t}}, y_{j_1,j_2,\dots,j_{t-1},j_{t}})> \frac{1}{3^{t-1}R^{t-1}}r^k$$
if $(i_1,i_2,\dots,i_{t-1})\neq (j_1,j_2,\dots,j_{t-1})$ and for every $i_{t},j_{t}\in \{1,2\}$. Let us do this until $2^t>M_1$ for some  $t\in \mathbb{N}$.

Since $r\leq \frac{1}{6}$, there exists $s\in \mathbb{N}$ such that $3r^s<\frac{1}{3^tR^t}$. Therefore,
$$3r^{k+s}<\frac{1}{3^tR^t}r^k<d(y_{i_1,i_2,\dots,i_{t-1},i_{t}},y_{j_1,j_2,\dots,j_{t-1},j_{t}})$$
for every $(i_1,i_2,\dots,i_{t})\neq (j_1,j_2,\dots,j_{t})$. Let $A_{k+s}$ be the $r^{k+s}$-approximation of $X$ in the construction of the hyperbolic approximation
$\Gamma'$. Thus, for every $(i_1,i_2,\dots,i_{t})$, $A_{k+s}\cap B(y_{i_1,i_2,\dots,i_{t}},r^{k+s})$ contains at least one point, $z_{i_1,i_2,\dots,i_{t}}$, and $z_{i_1,i_2,\dots,i_{t}}\neq z_{j_1,j_2,\dots,j_{t}}$ for every $(i_1,i_2,\dots,i_{t})\neq (j_1,j_2,\dots,j_{t})$.  Then, we obtain a set of $2^t$ points, $\mathcal{Z}=\{z_1,z_2,\dots,z_{2^t}\}\subset A_{k+s}$.

Let us see that for every $(i_1,i_2,\dots,i_{t})$, $B(z_{i_1,i_2,\dots,i_{t}},2r^{k+s})\subset B(x,2r^k)$. It suffices to see that for every $p\in B(z_{i_1,i_2,\dots,i_{t}},2r^{k+s})$,
$$
\begin{aligned}
d(p,x)
&\leq  d(p,z_{i_1,i_2,\dots,i_{t}})+d(z_{i_1,i_2,\dots,i_{t}},y_{i_1,i_2,\dots,i_{t}}) +
d(y_{i_1,i_2,\dots,i_{t}},y_{i_1,i_2,\dots,i_{t-1}})+\cdots + d(y_{i_1,i_2},y_{i_1}) +d(y_{i_1},x)
\\
& < 2r^{k+s} +r^{k+s} + \frac{1}{3^{t-1}R^{t-1}}r^k + \cdots + \frac{1}{3R}r^k + r^k<
\Big(\frac{1}{3^{t}R^{t}} + \frac{1}{3^{t-1}R^{t-1}} + \cdots + \frac{1}{3R} +1\Big)r^k
\\
&=
\frac{1-\frac{1}{3^{t+1}R^{t+1}}}{1-\frac{1}{3R}}\,r^k= \frac{3^{t+1}R^{t+1}-1}{3^{t}R^{t}(3R-1)}\,r^k< \frac{3R}{3R-1}\,r^k< 2r^k.
\end{aligned}
$$

Let us consider the truncated hyperbolic approximation  $\Gamma''=\mathcal{H}^t(\partial_\infty \Gamma)$ with parameter $r':=r^s$ and such that $A'_k:=A_{s\cdot k}$. Since these are two hyperbolic approximations of the same metric space, by Theorem
\ref{th:quasi-isometric}, $\Gamma''$ is quasi-isometric to $\Gamma'$ and by Lemma \ref{l:quasi-isometric}, it suffices
to check that $h(\Gamma'')>0$.

Consider $f:V(\Gamma'') \to \mathbb{R}$ such that for every $v\in V'_k$, $f(v)=k$.
It is immediate to check that $|\nabla_{a_ia_{i+1}}f|\leq 1$ for every pair of vertices $a_i,a_{i+1}\in V(\Gamma'')$.
Let $v$ be any vertex in $\Gamma''$ and suppose that $B(v)=B(x,2(r')^k)$.
Since $\partial_\infty \Gamma=\partial_\infty \Gamma''$, we have $|A'_m \cap \bar{B}(x',4(r')^m)|< M_1$ for any $x'\in X$ and $A'_m \neq \emptyset$.
Therefore, there are less than $M_1$ vertices in $V'_k$ adjacent to $v$.
Consider two vertices $w,w'$ in $V'_{k-1}$ adjacent to $v$.
If $B(w)=B(p,2(r')^{k-1})$ and $B(w')=B(q,2(r')^{k-1})$, then $B(x,2(r')^k)\subset B(p,2(r')^{k-1})\cap B(q,2(r')^{k-1})$ and $d(p,q)< 4(r')^{k-1}$.
Therefore, it follows that there exist at most $M_1$ vertices in $V'_{k-1}$ adjacent to $v$. Also, as we saw above, there exist $n\geq 2^t>M_1$ vertices in $V_{k+1}'$ adjacent to $v$. Thus, $(\Delta f)(a_i)\geq \frac{n-M_1}{n+2M_1} \geq \frac{2^t-M_1}{2^t+2M_1}>0$.

Therefore, by Proposition \ref{p:sufficient}, $h(\Gamma'')>0$.

Now, suppose $(\partial_\infty \Gamma,d)$ is not uniformly perfect and let $\Gamma'$ be the truncated hyperbolic approximation of  $(\partial_\infty \Gamma,d)$ with parameter $r$ described above. Then, for every $n\in \NN$ there exist $0<\varepsilon<\frac{1}{n}$ and $x_n\in \partial_\infty \Gamma$ such that $\bar{B}(x_n,\varepsilon)\backslash \bar{B}(x_n,r^{n+2}\varepsilon)=\emptyset$.

Claim: For every $k \in \NN$ such that $2r^{n+2}\varepsilon < r^k < r\varepsilon$  there is a  vertex $v_k\in V_k$ such that $x_n\in B(v_k)$.
Furthermore, $A_{k}\cap \bar{B}(x_n,\varepsilon)=\{a_{k}\}$ and $N(v_k)\cap V_k=\emptyset$.

Let us prove this claim.
Since $A_k$ is $r^k$-dense and $r^k < r\varepsilon$ there must be a point $a_k$ in $A_k\cap B(x_n,r\varepsilon)$.
Since, $\bar{B}(x_n,\varepsilon)\backslash \bar{B}(x_n,r^{n+2}\varepsilon)=\emptyset$, it follows that $a_k\in A_k\cap \bar{B}(x_n,r^{n+2}\varepsilon)$.
Therefore, $v_k$ with $B(v_k)=B(a_k,2r^k)$ satisfies the first part of the claim.
Also, since $A_k$ is $r^k$-separated and $r^k> 2r^{n+2}\varepsilon$, then $A_k\cap \bar{B}(x_n,r^{n+2}\varepsilon)=\{a_k\}$.
Since $\bar{B}(x_n,\varepsilon)\backslash \bar{B}(x_n,r^{n+2}\varepsilon)=\emptyset$,
we have $A_{k}\cap \bar{B}(x_n,\varepsilon)=\{a_{k}\}$.
Seeking for a contradiction, assume that there exists $w_k\in V_k$ with $w_k\in N(v_k)$. Let $B(w_k)=B(a'_k,2r^k)$ for some $a_k\neq a'_k\in A_k$. Thus,
$d(a'_k,a_k)\leq 4r^k$ and  $d(a'_k,x_n)\leq 4r^k + r^{n+2}\varepsilon < (4r+r^{n+2})\varepsilon<\varepsilon$.
Therefore, $a'_k\in \bar{B}(x_n,\varepsilon)$, contradicting $A_{k}\cap \bar{B}(x_n,\varepsilon)=\{a_{k}\}$.
Hence, the claim holds.

Let $k_1=\min\{k\, | \, r^k < r\varepsilon\}$ and $k_2=\max\{k \, | \, 2r^{n+2}\varepsilon < r^k\}$.
Notice that since $r^{k_1}<r\varepsilon \leq r^{k_1-1}$, then  $r^{k_1+n+1}<r^{n+2}\varepsilon \leq r^{k_1+n}$ and since $2<6 \leq r^{-1}$,
we have $2r^{n+2}\varepsilon < r^{k_1+n-1}$. Therefore, $k_2-k_1\geq n-1$.

Assume that $k_1<k<k_2$ and $v\in V_{k-1}\setminus\{v_{k-1}\}$, and consider $a$ with $B(v)=B(a,2r^{k-1})$.
Since $k_1 \le k-1<k_2$, we have $2r^{n+2}\varepsilon < r^{k-1} < r\varepsilon$.
Seeking for a contradiction, assume that $v\in N(v_k)$.
Thus,
$B(a_k,2r^{k})\subseteq  B(a,2r^{k-1})$,
$d(a,a_{k}) < 2r^{k-1}$ and $d(a,x_n) < 2r^{k-1} + r^{n+2}\varepsilon < (2r+r^{n+2})\, \varepsilon<\varepsilon$.
Therefore, $a \in \bar{B}(x_n,\varepsilon)$, contradicting $A_{k-1}\cap \bar{B}(x_n,\varepsilon)=\{a_{k-1}\}$.
Hence, $v\notin N(v_k)$.
A similar argument shows that if $v\in V_{k+1}\setminus\{v_{k+1}\}$, then $v\notin N(v_k)$.
Recall that, since $r\leq \frac{1}{6}$, every vertex in $V_m$ has a neighbor in $V_{m+1}$ for any $m$.
These facts and the claim give $N(v_k)=\{v_{k-1},v_{k+1}\}$ for every $k_1<k<k_2$.

Let $B_n=\{v_k | \, 2r^{n+2}\varepsilon < r^k < r\varepsilon \}=\{v_k \}_{k=k_1}^{k_2}$. Thus, $|B_n|=k_2-k_1+1\geq n$.
Propositions \ref{Prop:uniform} and \ref{Prop: bounded} give that there exists $\mu\in\NN$ such that $\G'$ is $\mu$-uniform.
Since $N(v_k)=\{v_{k-1},v_{k+1}\}$ for every $k_1<k<k_2$, we have $|\partial B_n|\leq 2\mu$.
Therefore, $\frac{|\partial B_n|}{|B_n|}\leq \frac{2\mu}{n}$ with $n$ arbitrarily large and we conclude $h(\Gamma')=0$.
Hence, by Theorem \ref{th:quasi-isometric} and Lemma \ref{l:quasi-isometric}, $h(\Gamma)=0$.
\end{proof}

\section{Hyperbolic manifolds}

Let us recall the following definition from \cite{K}.
Let $X$ be a complete Riemannian manifold and denote by $d$ the induced metric.
Given any $\varepsilon$-approximation $A_\varepsilon$ of $X$,
the graph $\Gamma_{A_\varepsilon}=(V,E)$ with $V=A_\varepsilon$ and $E:=\{xy \, | \, x,y \in A_\varepsilon \mbox{ with } 0<d(x,y)\leq 2\varepsilon\}$ is called an $\varepsilon$-\emph{net}.

\begin{proposition}\cite[Lemma 4.5]{K} \label{Prop: Kanai_net} Suppose that $X$ is a complete Riemannian manifold with bounded local geometry and let $\Gamma$ be an $\varepsilon$-net in $X$.
Then, $h(X)>0$ if and only if $h(\Gamma)>0$.
\end{proposition}

Note that the results in \cite{K} require $M$ to have positive injectivity radius and a lower bound on its Ricci curvature instead of bounded local geometry,
but the proofs in \cite{K} just use that there are uniform lower and upper bounds for the volume of the balls $B(x,r)$ which do not depend on $x\in M$ for $0<r<r_0$
(and we have these uniform bounds with bounded local geometry).
Hence, the results in \cite{K} also hold with the weaker hypothesis of bounded local geometry.

\begin{proposition}\cite[Lemma 2.5]{K} \label{Prop: net_qi}
Suppose that $X$ is a complete Riemannian manifold with bounded local geometry and let $\Gamma$ be an $\varepsilon$-net in $X$.
Then, $X$ and $\Gamma$ are quasi-isometric.
\end{proposition}

\begin{theorem}\cite[p.88]{GH}\label{th: stability_hyp}
If $f:X \rightarrow Y$ is a quasi-isometry between geodesic spaces, then $X$ is hyperbolic if and only if $Y$ is hyperbolic.
\end{theorem}

Let us denote by $\mathcal{H}$ the Hausdorff distance.
Recall that a \emph{quasi-geodesic} is a quasi-isometric embedding $\sigma:I \rightarrow X$ where $I$ is an interval $I\subseteq \RR$.

\begin{theorem}\cite[Chapter III.H, Theorem 1.7 (Geodesic stability)]{BH}\label{th: stability_geod}
For all $\delta,\beta\ge 0$, $\alpha\geq 1$, there exists a constant $R=R(\delta, \alpha, \beta)$ with the following property:

If $X$ is a $\delta$-hyperbolic geodesic space, $\sigma$ is an $(\alpha,\beta)$-quasi-geodesic in $X$ and $[pq]$ is a geodesic segment joining the endpoints of $\sigma$,
then $\mathcal{H}([pq],\sigma) \le R$.
\end{theorem}

Mario Bonk proved in \cite{Bo} that geodesic stability is, in fact, equivalent to hyperbolicity in geodesic spaces.

\begin{remark} \label{r:gs}
If $X$ is proper, then every quasigeodesic ray finishes at a point in $\p_\infty X$
and the conclusion of Theorem \ref{th: stability_geod} also holds if $q\in \partial_\infty X$ and $[pq]$ is a geodesic ray, by taking the limit.
\end{remark}

\begin{proposition}\label{Prop: qi-pole}
Suppose $X,Y$ are proper hyperbolic geodesic spaces and $f:X \to Y$ is a quasi-isometry.
If $X$ has a pole in $v$, then $Y$ has a pole in $f(v)$.
\end{proposition}

\begin{proof}
Consider constants $\a,\b,\e$ such that $f$ is an $\varepsilon$-full $(\alpha,\beta)$-quasi-isometric embedding.
Since $X$ has a pole $v$, there is a constant $M$ such that for every $x\in X$, there is a geodesic ray $\sigma$ with $\sigma(0)=v$ and $d(x,\sigma)\leq M$.
Also, for every $y\in Y$ there is some $x_0\in X$ such that $d(y,f(x_0))\leq \varepsilon$.
Let $\s_0$ be a geodesic ray with $\sigma_0(0)=v$ and $d(x,\sigma_0)\leq M$.
Therefore, $d(y,f\circ \sigma_0)\leq \alpha M+\beta+\varepsilon$.

Since $f\circ \sigma_0$ is an $(\alpha,\beta)$-quasi-geodesic ray with $(f\circ \sigma_0)(0)=v$, by Remark \ref{r:gs},
there is a geodesic ray $\sigma'$ with $\sigma'(0)=f(v)$ such that $f\circ \sigma_0$ lies in the $R$-neighborhood of $\sigma'$,
where $R$ is a constant which just depends on $\a$, $\b$ and the hyperbolicity constant of $Y$.
Therefore, $d(y,\sigma')\leq \alpha M+\beta+\varepsilon+R$ and $Y$ has a pole in $f(v)$.
\end{proof}

\begin{theorem} \cite[Chapter III.H, Theorem 3.9]{BH} \label{Th: homeom_boundary} Let $X,X'$ be proper hyperbolic geodesic spaces.
If $f:X\to X'$ is a quasi-isometry, then the canonically induced map in the boundary $\partial f:\partial_\infty X \to \partial_\infty X'$ is a homeomorphism.
\end{theorem}

\begin{proposition}\cite[Proposition 6.2]{GH} \label{Prop: compact} If $X$ is a proper hyperbolic geodesic space, then $\partial_\infty X$ with any visual metric is compact.
\end{proposition}

It is a basic fact of general topology that any homeomorphism $g:Z \to Z'$ between compact metric spaces is a uniform homeomorphism
(i.e., $g$ and $g^{-1}$ are uniformly continuous).
Therefore, the induced map $\partial f$ in Theorem \ref{Th: homeom_boundary} is a uniform homeomorphism.

Let us recall the following definition from \cite{BS}.  A map $f:X \to Y$ between metric spaces is
called \emph{quasi-symmetric} if it is not constant and if there
is a homeomorphism $\eta:[0,\infty) \to [0,\infty)$ such that from
$d_X(x,a)\leq t d_X(x,b)$ it follows that $d_Y(f(x),f(a))\leq
\eta(t)d_Y(f(x),f(b))$ for any $a,b,x\in X$ and all $t\geq 0$.
The function $\eta$ is called the \emph{control function} of $f$.

A quasi-symmetric map is said to be \emph{power quasi-symmetric}
or \emph{PQ-symmetric}, if its control function is of the form
\[\eta(t)= q \max\{t^p,t^{\frac{1}{p}}\}\] for some $p,q\geq 1$. For a characterization of this property, see \cite{M1}.

\begin{theorem}\cite[Theorem 5.2.15]{BS} \label{Th: pq-symmetric} Let $f : X \to Y$ be a quasi-isometric map of proper hyperbolic
geodesic spaces. Then, the naturally induced  map
$\partial f : \partial_\infty X \to \partial_\infty Y$  is
PQ-symmetric with respect to any visual metrics.
\end{theorem}

\begin{proposition}\label{Prop: perfect} If $g:Z \to Z'$ is a PQ-symmetric uniform homeomorphism between the metric spaces $Z,Z',$ and $Z$ is uniformly perfect, then $Z'$ is uniformly perfect.
\end{proposition}

\begin{proof} Suppose $Z$ is $S$-uniformly perfect (with constant $\varepsilon_0$).
By uniform continuity of $f^{-1}$, let $\varepsilon'_0$ be such that
if $d'(f(x_1),f(x_2))\leq \varepsilon'_0$, then $d(x_1,x_2)<\frac{\varepsilon_0}{S}$.
Let $x'$ be any point in $Z'$ and let $0<\varepsilon'\leq \varepsilon'_0$.
Let us see that there is some point $y' \in Z'$ such that $\frac{\varepsilon'}{qS^{2p}} < d'(x',y')\leq \varepsilon'$.

Let $x=f^{-1}(x')$. Thus, there is a point $y_1\in Z$ such that $\frac{\varepsilon_0}{S}<d(x,y_1)\leq \varepsilon_0$.
Since $\frac{\varepsilon_0}{S}<d(x,y_1)$, if $y'_1=f(y_1)$ it follows that $\varepsilon'_0 < d'(x',y'_1)$.

Since $Z$ is uniformly perfect, there is a point $y_2$ such that $\frac{\varepsilon_0}{S^2} < d(x,y_2)\leq \frac{\varepsilon_0}{S}$. In fact, there is a sequence of points $(y_k)$ such that $\frac{\varepsilon_0}{S^k} < d(x,y_k)\leq \frac{\varepsilon_0}{S^{k-1}}$. Clearly, $(y_k)\to x$ and if $y'_k=f(y_k)$, then $(y'_k)\to x'$.

Since $f$ is PQ-symmetric, there are constants $p,q \ge 1$ such that $d(x,a)\leq t d(x,b)$ implies that
$$
d'(f(x),f(a))\leq q \max\{t^p,t^{\frac{1}{p}}\} d'(f(x),f(b))
$$
for any $a,b,x\in X$ and all $t\geq 0$.
Since $d(x,y_k)< S^2 d(x,y_{k+1})$ for every $k\in \mathbb{N}$,
we have $d'(x',y'_k)\le q (S^2)^{p} d'(x',y'_{k+1})$ and $d'(x',y'_{k+1}) \ge \frac{1}{qS^{2p}}d'(x',y'_k)$.

Seeking for a contradiction assume that $\bar{B}(x',\varepsilon')\backslash \bar{B}(x',\frac{\varepsilon'}{qS^{2p}})=\emptyset$.
Then, since $d'(x',y'_1)>\varepsilon'_0\geq \varepsilon'$ and $(y'_k)\to x'$, there is some $n\in \NN$ such that $\varepsilon' < d'(x',y'_n)$ and $d'(x',y'_{n+1}) \le \frac{\varepsilon'}{qS^{2p}}$.
Therefore, $d'(x',y'_{n+1})<\frac{1}{qS^{2p}}d'(x',y'_n)$ leading to contradiction.
Hence, there exists $y' \in \bar{B}(x',\varepsilon')\backslash \bar{B}(x',\frac{\varepsilon'}{qS^{2p}})$ and $Z'$ is uniformly perfect.
\end{proof}

\begin{lemma}\cite[Lemma 2.3]{K} \label{lema: uniform-net} Every $\varepsilon$-net in a complete Riemannian manifold with bounded local geometry is uniform.
\end{lemma}

\begin{theorem}\label{th: isoperimetric_manifolds}
Let $X$ be a non-compact complete Riemannian manifold with bounded local geometry.
Assume that $X$ is hyperbolic and has a pole.
Then, $h(X)>0$ if and only if $\partial_\infty X$ is uniformly perfect.
\end{theorem}

\begin{proof} Let $\Gamma$ be an $\varepsilon$-net in $X$. Then, by Proposition \ref{Prop: Kanai_net}, $h(X)>0$ if and only if $h(\Gamma)>0$.

Note that $X$ is a proper geodesic space since it is a complete Riemannian manifold.

By Proposition \ref{Prop: net_qi}, $X$ and $\Gamma$ are quasi-isometric.
Thus $\Gamma$ is hyperbolic by Theorem \ref{th: stability_hyp}.
Therefore, by Theorem \ref{Th: homeom_boundary} and
Proposition \ref{Prop: compact}, there is a unifom homeomorphism between $\partial_\infty X$ and $\partial_\infty \Gamma$ and
by Theorem \ref{Th: pq-symmetric}, this homeomorphism  and its inverse are PQ-symmetric.

By Proposition \ref{Prop: perfect}, $\partial_\infty \Gamma$ is uniformly perfect if and only if $\partial_\infty X$ is uniformly perfect.
By Proposition \ref{Prop: qi-pole}, $\Gamma$ has a pole and by Lemma \ref{lema: uniform-net}, $\Gamma$ is uniform.
Thus, by Theorem \ref{t:iigraph}, $h(\Gamma)>0$ if and only if $\partial_\infty X$ is uniformly perfect.
\end{proof}

Let us recall the following definition from \cite{Cao}.

\begin{definition} A complete manifold
(or graph) $X$ is said to have a quasi-pole in a compact subset
$\Omega \subset X$ if there
exists $C > 0$ such that each point of $X$ lies in a $C$-neighborhood of some geodesic
ray emanating from $\Omega$.
\end{definition}

The main result in \cite{Cao} is the following:

\begin{theorem}\label{Theorem: Cao} Let $X$ be a non-compact complete manifold (or a graph) which
admits a quasi-pole and has bounded local geometry. Suppose that $X$ is Gromov-hyperbolic
and the diameters of the connected components of $\partial^g X$ have a positive
lower bound (with respect to a fixed Gromov metric). Then $h(X)>0$.
\end{theorem}

Let us see that this is a particular case of Theorems \ref{t:iigraph} and \ref{th: isoperimetric_manifolds}.

\begin{proposition}\label{Prop: quasi-pole} A proper hyperbolic geodesic space $X$ has a quasi-pole if and only if it has a pole.
\end{proposition}

\begin{proof} The if part is immediate.
Suppose that $X$ has a quasi-pole in the compact $\Omega$ and define $D=\diam \Omega < \infty$.
Given any pair of points $v,w\in \Omega$, consider any pair of  geodesic rays $\gamma_v, \gamma_w$ starting at $v$ and $w$, respectively, with $[\gamma_v]=[\gamma_w]$ in $\partial^gX$.
Since $L([vw]) \le D$, one can check that $[vw] \cup \g_w$ is a $(1,2D)$-quasigeodesic (see \cite[Lemma 2.14]{PRT1} for a proof).
By Remark \ref{r:gs}, there is a constant $R$ (which just depend on $D$ and the hyperbolicity constant of $X$) such that $\mathcal{H}(\g_v,[vw] \cup \g_w) \le R$.
Hence, $\mathcal{H}(\g_v, \g_w) \le R+D$.

Fix any $v\in \Omega$.
Since $X$ has a quasi-pole in $\Omega$, there exists some constant $C>0$ such that for every $x\in X$, there is a geodesic ray, $\gamma$ starting at some point $w_x\in \Omega$ and such that $d(x,\gamma)<C$.
Let $\gamma_v$ be the geodesic ray starting at $v$ such that $[\gamma]=[\gamma_v]$. Therefore, $\mathcal{H}(\g_v, \g) \le R+D$ and $d(x,\gamma_v)<C+R+D$. Hence, $v$ is a pole.
\end{proof}

\begin{proposition}\label{Prop: pseudo-regular} Given a metric space $Z$, if the diameters of the connected components of $Z$ have a positive lower bound, then $Z$ is uniformly perfect.
\end{proposition}

\begin{proof} Suppose that for every connected component $C$ of $Z$, $\diam(C)>r>0$, and let $p\in Z$.
Then, if $C_p$ is the connected component of $Z$ containing $p$, for every $0<s\leq \frac{r}{2}$ there exist some $q_s\in C$ such that $d(p,q)=s$.
Otherwise, $B(p,s)$ and $Z\backslash B(p,s)$ would define an open partition on $C_p$ and the connected component of $Z$ containing $p$ would have diameter at most $r$.

Therefore, $Z$ is $S$-uniformly perfect for every $S>1$ and any $\varepsilon_0\leq \frac{r}{2}$.
\end{proof}

\begin{remark} \label{r:Cao}
Cao's result (Theorem \ref{Theorem: Cao}) is a particular case of Theorems \ref{t:iigraph} and \ref{th: isoperimetric_manifolds}.
By Proposition \ref{Prop: quasi-pole}, the existence of a quasi-pole is equivalent to the existence of a pole.
Finally, by Proposition \ref{Prop: pseudo-regular}, the condition on the components of the boundary implies that the boundary is uniformly perfect.
\end{remark}

We finish this section with two corollaries, following \cite{Cao}.

Let $X$ be any non-compact hyperbolic complete Riemannian manifold or graph.
Given any continuous function $g$ on $\p_\infty X$, consider the Dirichlet problem at infinity
$$
\left\{
\begin{array}{cc}
    \Delta u = 0 & \text{on }\, X, \\
    \lim_{x\to \xi} u(x)=g(\xi) & \qquad \text{for }\, \xi \in \p_\infty X. \\
  \end{array}
\right.
$$

\begin{corollary} \label{c1}
Let $X$ be a complete Riemannian manifold or graph.
Assume that $X$ is non-compact and hyperbolic, and that it has bounded local geometry and a pole.
Then the Dirichlet problem at infinity is solvable on $X$ and, hence, $X$ admits infinitely many linearly independent bounded non-constant harmonic functions.
\end{corollary}

In order to study the space of normalized minimal positive harmonic functions
on $X$, we need to consider the Martin boundary $X$ (see \cite{AS}).
Using Theorems \ref{t:iigraph} and \ref{th: isoperimetric_manifolds} and the work of Ancona (\cite{A1}, \cite{A2}, \cite{A3}), we can extend a theorem of Anderson and Schoen \cite{AS} (and the work of Kifer \cite{Ki}):

\begin{corollary} \label{c2}
Let $X$ be a complete Riemannian manifold or graph.
Assume that $X$ is non-compact and hyperbolic, and that it has bounded local geometry and a pole.
Then the Martin boundary of $X$ is homeomorphic to $\p_\infty X$.
\end{corollary}

Corollaries \ref{c1} and \ref{c2} follow directly from Theorems \ref{t:iigraph} and \ref{th: isoperimetric_manifolds} since for $X$ hyperbolic with $\l_1(X) > 0$
(recall that $\l_1(X) \ge \frac14 h(X)^2$) and bounded local geometry, they are special cases of known results in \cite{A2} and \cite{A3}.

Using Theorem \ref{t:iitreefinal}, instead of Theorems \ref{t:iigraph} and \ref{th: isoperimetric_manifolds}, we obtain the following result for trees (with weaker hypotheses).

\begin{corollary} \label{c3}
Let $(T,v)$ be a complemented unbounded rooted tree with $(T_\infty,v)$ pseudo-regular.
Then the Dirichlet problem at infinity is solvable on $T$ and the Martin boundary of $T$ is homeomorphic to $\p_\infty T=end(T,v)$.
\end{corollary}

In \cite{Cao} appear several sufficient conditions in order to guarantee that a Riemannian manifold is hyperbolic.
See also \cite{T} in the case of Riemannian surfaces.

\section{Decompositions and isoperimetric inequalities}

In this chapter we are interested in the relations between the isoperimetric inequality in a graph and its subgraphs.
In this way, we obtain global information from local information.

Given a graph $\G$ and positive constants $R,r,$ we say that a family of subgraphs $\{\G_{\!s} \}_{s\in S}$ of $\G$ is an \emph{$(R,r)$-decomposition} of $\G$ if:

$(1)$ $\cup_{s\in S} \G_{\!s} = \G $,

$(2)$ $\G_{\!s}\cap \G_{\!r}$ is either a (finite or not) set of vertices or the empty set and $\G_{\!s}\setminus \G_{\!r} \neq \emptyset$ for each $s\neq r$,

$(3)$ there is a partition $\{S_1,S_2\}$ of $S$ such that

$\quad (3.1)$ $S_1 \neq \emptyset$ and $h(\G_{\!s}) \ge r$ for every $s\in S_1$,

$\quad (3.2)$ if $V_s=\G_{\!s}\cap \big( \cup_{t\in S_1} \G_t \big)$, $W_s=\cup_{v\in V_s} \bar{B}_{\G_{\!s}}(v,R)$ and $\{\G_{\!s,j}\}_{j\in J_s}$ are the connected components of the subgraph of $\G_{\!s}$ induced by
$V(\G_{\!s}) \setminus W_s$,
then $V_s \neq \emptyset$ and $h(\G_{\!s,j}) \ge r$ for every $s\in S_2$ and every $j \in J_{s}$, where $\bar{B}_{\G_{\!s}}$ denotes the closed ball in $\G_{\!s}$.

\begin{theorem} \label{t:Tdec1}
If a $\mu$-uniform graph $\G$ has an $(R,r)$-decomposition, then
$$
h(\G)
\ge \frac{r^2( \mu -1)}{ (\mu^{R+1}-1 )( \mu + r)^2+ 2r \mu( \mu-1)}\,.
$$
\end{theorem}

\begin{proof}
Let $\{\G_{\!s} \}_{s\in S}$ be an $(R,r)$-decomposition of $\G$.
Fix any non-empty finite subset $A$ of vertices in $\Gamma$, and let us define $A_s:=A \cap \G_s$ for each $s\in S$.
Since $A$ is finite and $\G$ is a uniform graph, the sets
$S_1^A:=\{s\in S_1\,|\; A_s\neq \emptyset\}$ and $S_2^A:=\{s\in S_2\,|\; A_s\neq \emptyset\}$ are finite.

Denote by $\p' \!A$ the subset of $\p A$ given by $\p' \!A = \p A \setminus \cup_{s\in S_2} ( \G_{\!s}\setminus W_s)= \p A \setminus \cup_{s\in S_2} \cup_{j\in J_s} \G_{\!s,j}$.

If $D$ is a subset of $V(\G_s)$, let us denote by $\p_{\,\G_{\!s}} D$ the boundary of $D$ in the graph $\G_{\!s}$.
One can check that $\p_{\,\G_{\!s}} A_s \subseteq \G_{\!s} \cap \p A$.
Note that it is possible to have $\p_{\,\G_{\!s}} A_s \neq \G_{\!s} \cap \p A$ (if $\G_{\!s}$ is a finite graph for some $s\in S_2$ and
$A=\p \,V(\G_{\!s})\neq \emptyset$,
then $A_s=\emptyset$, $\p_{\,\G_{\!s}} A_s=\emptyset$ and $\G_{\!s} \cap \p A \neq \emptyset$).
We have that $\p_{\,\G_{\!s}} A_s$ is a non-empty set for each $s \in S_1^A$, since $h(\G_{\!s}) \ge r$ gives that $\G_{\!s}$ is an infinite graph.
For every $s \in S_1^A$, the inequality
$|A_s| \le h(\G_{\!s})^{-1}|\p_{\,\G_{\!s}} A_s| \le r^{-1}|\p_{\,\G_{\!s}} A_s|$
holds (this inequality trivially holds for every $s \in S \setminus (S_1^A \cup S_2^A)$).
Since $\cup_{s\in S_1} \, \p_{\,\G_{\!s}} A_s \subseteq \p' A$
and $\G$ is a $\mu$-uniform graph, a vertex belongs at most to $\mu$ subgraphs in $\{\G_{\!s} \}_{s\in S}$,
$\sum_{s\in S_1}|\p_{\,\G_{\!s}} A_s| \le \mu\,|\p' \!A|$ and
$\sum_{s\in S_1}| A_s| \le r^{-1}\sum_{s\in S_1}|\p_{\,\G_{\!s}} A_s| \le \mu \, r^{-1}|\p' \!A|$.

Let us bound now $| \cup_{s\in S_2} A_s|$.

Assume first that $A_s \subseteq W_s$ for each $s\in S_2$.
Then, for every $v\in A_s$ with $s\in S_2$, there exists $v'\in V_s$ such that $d_{\, \G_{\!s}}(v,v')\leq R$.
Let $\g_v$ be a geodesic in $\G_s$ joining $v$ and $v'$.
We have either that there exists a vertex $v^* \in \g_v\cap \p' \!A$ or $\g_v\cap V(\G_{\!s})\subseteq A$,
in this last case we define $v^* =v'$.
Hence, $d_\G(v,v^*) \le R$ and $v^* \in W_s$ in any case.
Now fix $v\in A_s$ with $s\in S_2$.
Note that if $w^*=v^*$ for some vertex $w\in A_t$ with $t\in S_2$, then $d_{\,\G}(w,v^*) \le R$.
Since $\G$ is $\mu$-uniform, we have that
$$
\big| \{w\in \cup_{s\in S_2} A_s\,|\; w^*=v^* \} \big|
\le \sum_{j=0}^{R} \mu^j
= \frac{\mu^{R+1}-1}{\mu-1}\,.
$$

If we define $C:=\cup_{s\in S_2}\{v\in A_s\,|\; v^*\in \p A \}=\cup_{s\in S_2}\{v\in A_s\,|\; v^*\in \p' \!A \}$ and $D:=\cup_{s\in S_2} A_s\setminus C$, then
$$
\begin{aligned}
| C|
& \le \frac{\mu^{R+1}-1}{\mu-1} |\p' \!A|,
\\
| D|
& \le \frac{\mu^{R+1}-1}{\mu-1} \big| \{v'\,|\; v\in D \} \big|
\le \frac{\mu^{R+1}-1}{\mu-1} \big| \cup_{s\in S_1} A_s \big|
\\
& \le \frac{\mu^{R+1}-1}{\mu-1} \sum_{s\in S_1} | A_s |
\le \frac{\mu^{R+1}-1}{\mu-1} \mu \,r^{-1}|\p' \!A|,
\\
| A|
& \le \sum_{s\in S_1} | A_s | + \big| \cup_{s\in S_2} A_s \big|
= \sum_{s\in S_1} | A_s | + |C| + |D|
\\
& \le \mu \,r^{-1}|\p' \!A| + \frac{\mu^{R+1}-1}{\mu-1} |\p' \!A|
+ \frac{\mu^{R+1}-1}{\mu-1} \mu \,r^{-1}|\p' \!A|
\\
& = \frac{\big(\mu^{R+1}-1 \big) \big( \mu +r \big) + \mu \big( \mu -1\big)}{r\big( \mu -1 \big)} \, |\p' \!A|,
\end{aligned}
$$
and this gives the desired inequality.

Consider now the general case and define
$$
A^*= A \setminus \cup_{s\in S_2} ( \G_{\!s}\setminus W_s),
\quad
A_s^{*}=A^{*} \cap \G_{\!s},
\quad
A^{**}=A \setminus A^*,
\quad
A_{s,j}^{**}=A^{**} \cap \G_{\!s,j}.
$$
We have proved
$$
| A^*|
\le  \frac{\big(\mu^{R+1}-1 \big) \big( \mu +r \big) + \mu \big( \mu -1\big)}{r\big( \mu -1 \big)}\, |\p' \!A^*|,
\qquad
\big| \cup_{s\in S_2} A_s^* \big|
\le \frac{\mu^{R+1}-1}{\mu-1} \big(\mu \,r^{-1}+1  \big)|\p' \!A^*|.
$$
Since $\p' \!A^* \subseteq \p A$, we deduce
$$
| A^*|
\le \frac{\big(\mu^{R+1}-1 \big) \big( \mu +r \big) + \mu \big( \mu -1\big)}{r\big( \mu -1 \big)}\, |\p A|,
\qquad
\big| \cup_{s\in S_2} A_s^* \big|
\le \frac{\mu^{R+1}-1}{\mu-1} \big(\mu \,r^{-1}+1  \big)|\p A|.
$$

If $v\in \p_{\,\G_{\!s,j}} A_{s,j}^{**} \setminus \p  A$, then $v\in A_s^{*}$.
Therefore,
$$
\begin{aligned}
| A^{**}|
& = \sum_{s\in S_2} \sum_{j\in J_s} | A_{s,j}^{**}|
\le \sum_{s\in S_2} \sum_{j\in J_s} r^{-1} |\p_{\,\G_{\!s,j}} A_{s,j}^{**}|
\\
& \le r^{-1} \Big( \sum_{s\in S_2} \sum_{j\in J_s} |\p A \cap \G_{\!s,j}| + \sum_{s\in S_2} \sum_{j\in J_s} |\p_{\,\G_{\!s,j}} A_{s,j}^{**}\setminus \p A | \Big)
\\
& \le r^{-1} \big( \mu \,|\p A | + \mu\,| \cup_{s\in S_2} A_s^{*} | \big)
\\
& \le  r^{-1}\Big( \mu + \mu\, \frac{\mu^{R+1}-1}{\mu-1}  \big(\mu \,r^{-1}+1  \big)\Big)|\p A|,
\end{aligned}
$$
and
$$
\begin{aligned}
| A|
& = | A^{*} | + | A^{**} |
\\
& \le \Big( \, \frac{ r(\mu^{R+1}-1 )( \mu + r)+ r\mu ( \mu-1)}{r^2( \mu -1)}+ \frac{ r \mu( \mu-1) + \mu (\mu^{R+1}-1 )( \mu + r)}{r^2( \mu -1)}\Big)|\p A|
\\
& = \frac{ (\mu^{R+1}-1 )( \mu + r)^2+ 2r \mu( \mu-1)}{r^2( \mu -1)}\,|\p A|,
\end{aligned}
$$
and this finishes the proof.
\end{proof}

We say that the $(R,r)$-decomposition $\{\G_{\!s} \}_{s\in S}$ of $\G$ is a \emph{strong $(R,r)$-decomposition}
if $V(W_s) = V(\G_{\!s})$ for every $s\in S_2$.

The first part of the argument in the proof of Theorem \ref{t:Tdec1} has the following consequence.

\begin{corollary} \label{c:Tdec1}
If a $\mu$-uniform graph $\G$ has a strong $(R,r)$-decomposition, then
$$
h(\G) \ge \frac{ r ( \mu -1 ) }{(\mu^{R+1}-1 ) ( \mu +r ) + \mu ( \mu -1)}\,.
$$
\end{corollary}

This last result can be used in order to show that the isoperimetric inequality in graphs does not imply hyperbolicity.
We just need two additional lemmas.
We say that a subgraph $\G$ of $G$ is \emph{isometric} if $d_{\G}(u,v)=d_{G}(u,v)$ for every $u,v\in V(\G)$.
Isometric subgraphs are very important in the study of hyperbolic graphs, as the following result shows.

\begin{lemma}\cite[Lemma 5]{RSVV}
\label{l:subgraph}
If $\G_0$ is an isometric subgraph of $\G$, then $\d(\G) \ge \d(\G_0)$.
\end{lemma}

\begin{lemma} \label{l:exa1}
For each uniform graphs $\G,G,$ with $h(G)>0$, there exists a uniform graph $\tilde{\G}$ such that $\G$ is an isometric subgraph of $\tilde{\G}$, $h( \tilde{\G} )>0$ and $\d( \tilde{\G} ) \ge \d(\G)$.
\end{lemma}

\begin{proof}
Let $\{G_k\}_{k\in \NN}$ be a sequence of graphs isometric to $G$ with $v_k \in V(G_k)$ for each $k \in \NN$.
If $V(\G)=\{w_k\}_{k\in \NN}$, then let $\tilde{\G}$ be the graph obtained from $\G$ and $\{G_k\}_{k\in \NN}$ by identifying $v_k$ with $w_k$ for each $k\in \NN$.
Since $V(\tilde{\G})=\cup_{k\in \NN} V(G_k)$, we have that $\{\{G_k\}_{k\in \NN},\{\G\}\}$ is a strong $(0,h(G))$-decomposition of $\tilde{\G}$.
Therefore, Corollary \ref{c:Tdec1} gives that $h(\tilde{\G})>0$.
If $\G,G$ are $\mu$-uniform, then $\tilde{\G}$ is $2\mu$-uniform.
One can check that $\G$ is an isometric subgraph of $\tilde{\G}$, and Lemma \ref{l:subgraph} gives $\d( \tilde{\G} ) \ge \d(\G)$.
\end{proof}

By taking any non-hyperbolic uniform graph $\G$, Lemma \ref{l:exa1} gives the following.

\begin{example} \label{exa1}
There exist non-hyperbolic uniform graphs with positive isoperimetric constant.
\end{example}

Given a subgraph $\G_0$ of the graph $\G$,
the \emph{interior Cheeger isoperimetric constant} of $\Gamma_0$ is defined to be
\[h^i(\Gamma_0) = \inf_A \frac{|\partial_{\,\G_0} A|}{|A|},\]
where $A$ ranges over all non-empty finite subsets of vertices in $\Gamma_0$ with $d(A,\p \G_0) \ge 2$, and
$\partial_{\,\G_0} A$ denotes the boundary of $A$ in the subgraph $\G_0$.
It is clear that $h^i(\Gamma_0) \ge h(\Gamma_0)$.
We also have the following result.

\begin{lemma} \label{l:i}
If $\G$ is any graph, then
$$
h(\Gamma) = \min \big\{ h^i(\Gamma_0)\,|\; \G_0 \,\text{ is a subgraph of }\,\G\, \big\}.
$$
\end{lemma}

\begin{proof}
Fix a subgraph $\G_0$ of $\G$ and a non-empty finite subset of vertices $A$ in $\Gamma_0$ with $d(A,\p \G_0) \ge 2$.
Since $\G_0 \subseteq \G$, we have $\partial_{\,\G_0} A \subseteq \partial A$.
Since $A \subset V(\G_0)$ and $d(A,\p \G_0) \ge 2$, if a vertex belongs to $\partial A$,
then it belongs to $\G_0$ and, consequently, it belongs also to $\partial_{\,\G_0} A$.
Hence, $\partial_{\,\G_0} A = \partial A$ and $h^i(\Gamma_0) \ge h(\Gamma)$.
Finally, if $\Gamma_0=\G$, then $h^i(\Gamma) = h(\Gamma)$ and the equality in the lemma holds.
\end{proof}

Lemma \ref{l:i} has the following consequence, which can be viewed as a kind of converse of Theorem \ref{t:Tdec1}.

\begin{theorem} \label{t:converse}
If $\G$ is a graph with a sequence of subgraphs $\{\G_n\}$ such that $\lim_{n\to\infty} h^i(\Gamma_n) = 0$, then $h(\Gamma)=0$.
\end{theorem}

\end{document}